\documentclass[english]{amsart}

\usepackage{a4wide,hyperref,url,amsmath,amssymb,amsthm,esint,graphicx}
\usepackage[capitalize]{cleveref}
\usepackage[shortlabels]{enumitem}

\hypersetup{pdfpagemode=UseNone, pdfstartview=FitH,
  colorlinks=true,urlcolor=red,linkcolor=blue,citecolor=black,
  pdftitle={Default Title, Modify},pdfauthor={Your Name},
  pdfkeywords={Modify keywords}}

\allowdisplaybreaks

\newtheorem{theorem}{Theorem}[section]
\newtheorem{lemma}[theorem]{Lemma}

\newtheorem{corollary}[theorem]{Corollary}

\theoremstyle{definition}
\newtheorem{definition}[theorem]{Definition}

\newtheorem{remark}[theorem]{Remark}

\crefname{corollary}{Corollary}{Corollaries}
\Crefname{corollary}{Corollary}{Corollaries}
\crefname{subsection}{Section}{Sections}
\Crefname{subsection}{Section}{Sections}
\crefname{equation}{}{}
\Crefname{equation}{Formula}{Formulas}
\crefname{enumi}{}{}

\newcommand\f\frac
\newcommand\tx\text
\newcommand\intd{\mathop{}\!\mathrm{d}}
\newcommand\M{\mathrm M}
\newcommand\Md{\mathrm M^{\mathrm d}}
\newcommand\I{\mathrm I}
\newcommand\ind[1]{1_{#1}}
\newcommand\lm[1]{\mathcal{L}(#1)}
\newcommand\lmb[1]{\mathcal{L}\biggl(#1\biggr)}
\newcommand\sm[1]{\mathcal{H}^{n-1}(#1)}
\newcommand\mb[1]{{\partial_*}{#1}} 
\newcommand\tb[1]{{\partial}{#1}} 

\newcommand\B{\mathcal B}
\newcommand\avint\fint
\DeclareMathOperator\sle{l}
\newcommand\loc{\text{\normalfont loc}}
\newcommand\tc[1]{\overline{#1}}
\newcommand\mc[1]{\tc{#1}^*}
\newcommand\ti[1]{\mathring{#1}}

\title{Weighted fractional Poincar\'e inequalities via isoperimetric inequalities}

\author{Kim Myyryl\"ainen}
\address{Department of Mathematics, Aalto University, P.O. Box 11100, FI-00076 Aalto, Finland}
\email{kim.myyrylainen@aalto.fi}
\thanks{The research was supported by the Academy of Finland. K.~Myyryl\"ainen and J.~Weigt have also been supported by the Magnus Ehrnrooth Foundation.}

\author{Carlos P\'erez}
\address{Department of Mathematics, University of the Basque Country, IKERBASQUE 
(Basque Foundation for Science) and
BCAM \textendash  Basque Center for Applied Mathematics, Bilbao, Spain}
\email{cperez@bcamath.org}
\thanks{%
C. P\'erez is supported by grant  PID2020-113156GB-I00, Spanish Government; by the Basque Government through grant IT1615-22
and the BERC 2014-2017 program  and by the BCAM Severo Ochoa accreditation CEX2021-001142-S, Spanish Government.%
}

\author{Julian Weigt}
\address{%
Mathematics Institute,
Zeeman Building,
University of Warwick,
Coventry CV4 7AL,
United Kingdom}
\email{julian.weigt@warwick.ac.uk}
\thanks{J.~Weigt is also supported by the European Union’s Horizon 2020 research and innovation programme (Grant agreement No. 948021).}

\thanks{%
We are very grateful to the Department of Mathematics at Aalto University for its support where the initiation of this research took place. In particular, we are very grateful to Juha Kinnunen for his support and for the discussions we had.%
}



\begin{document}

\maketitle

\begin{abstract}

Our main result is a weighted fractional Poincar\'e--Sobolev  inequality improving the celebrated estimate by Bourgain--Brezis--Mironescu. 
This also yields an improvement of the classical Meyers--Ziemer theorem in several ways.
The proof is based on a fractional isoperimetric inequality
and is 
new even in the non-weighted setting.
We also extend the celebrated Poincar\'e--Sobolev estimate with $A_p$ weights of Fabes--Kenig--Serapioni  by means of a fractional type result in the spirit of Bourgain--Brezis--Mironescu.
Examples are given to show that the corresponding $L^p$-versions of weighted Poincar\'e inequalities do not hold for $p>1$.

\end{abstract}

\section{Introduction}

The classical $(q,p)$-Poincar\'e--Sobolev inequality states that
\begin{equation}
\label{intro:claspoin}
\biggl(
\int_Q
\lvert f-f_Q \rvert^q
\intd x
\biggr)^\frac{1}{q}
\leq
C
\biggl(
\int_Q
|\nabla f|^p
\intd x
\biggr)^\frac{1}{p} ,
\end{equation}
where $1\leq p<n$, $q=\frac{np}{n-p}$, $f\in W_\loc^{1,p}(\mathbb{R}^n)$, $Q\subset\mathbb{R}^n$ is a cube and \(C\) is a dimensional constant.
In 2002, Bourgain, Brezis and Mironescu~\cite{bbm2002} proved the following fractional $(q,p)$-Poincar\'e inequality
\begin{equation}
\label{intro:fracpoin}
\biggl(
\avint_Q \lvert f-f_Q \rvert^q \intd x \biggr)^\frac{1}{q}
\leq
C
(1-\delta)^\f1p
\sle(Q)^\delta
\biggl(
\avint_Q
\int_Q
\f{|f(x)-f(y)|^p}{|x-y|^{n+\delta p}}
\intd y
\intd x
\biggr)^\frac{1}{p} ,
\end{equation}
where $\f12\leq\delta<1$, $1\leq p<\frac{n}{\delta}$, $q=\frac{np}{n-\delta p}$, $f\in L_\loc^1(\mathbb{R}^n)$, $Q\subset\mathbb{R}^n$ is a cube and \(C\) is a dimensional constant.
Note the factor $(1-\delta)^\f1p$ in front of the fractional term which balances the limiting behaviour of the right-hand side when $\delta\to1$.
In particular, it was shown by Brezis~\cite{brezis2002constant} that without the factor the right-hand side of~\cref{intro:fracpoin} is infinite for non-constant functions when $\delta=1$.
Moreover, Bourgain, Brezis and Mironescu~\cite{BBManotherlook2001} showed that with this factor the fractional term coincides with the $L^p$ norm of the gradient when $\delta\to1$.
This means that in the limit \cref{intro:fracpoin} turns into the classical Poincar\'e inequality~\cref{intro:claspoin}.
Later,
Maz'ya and Shaposhnikova~\cite{mazya2002} 
proved the corresponding inequality in $\mathbb{R}^n$.
They showed in $\mathbb{R}^n$ that the fractional term multiplied with $\delta^\f1p$ coincides with the $L^p$ norm of the function when $\delta\to0$.
For other limiting behaviour results, 
we refer to
Alberico, Cianchi, Pick and Slav\'{\i}kov\'{a}~\cite{AlbericoCianchi2020},
Brezis, Van Schaftingen and Yung~\cite{brezisSchaftingenYung2021}, 
Drelichman and Dur\'{a}n~\cite{DrelichmanRicardo2022}
and
Karadzhov, Milman and Xiao~\cite{KaraMilmanXiao2005}.

The existing proofs of the fractional Poincar\'e inequality apply Fourier analysis techniques~\cite{bbm2002}, Hardy type inequalities~\cite{mazya2002} or 
compactness arguments~\cite{ponce2004}.
We give a new direct and transparent proof
using
a relative isoperimetric inequality as our main tool.
We concentrate on the case $p=1$ in \cref{intro:fracpoin}.
Our approach is based on a new fractional type isoperimetric inequality in \cref{lem:cla_weakcharf} which 
can be seen as an improvement of the classical relative isoperimetric inequality, see \cref{rem:cla_weakcharf_rhs}.
To our knowledge this approach with isoperimetric inequalities has not been considered in the fractional case before.
This allows further investigation of the theory of fractional Poincar\'e inequalities.

It is known that the classical $(1,1)$-Poincar\'e inequality implies the classical $(q,p)$-Poincar\'e inequality.
We investigate this in the fractional setting with $A_p$ weights.
The strategy is to first show that the fractional $(1,1)$-Poincar\'e inequality implies the fractional $(1,p)$-Poincar\'e inequality in \cref{prop:(1_1)-to-(1_p)}. 
Then we apply a self-improving property and a fractional truncation method to obtain the fractional $(q,p)$-Poincar\'e inequality with $A_p$ weights, see \cref{selfIMproveBadConstant} and \cref{selfIMproveGoodConstant}.
This extends the fractional Poincar\'e inequality 
in Hurri-Syrj{\"a}nen, Mart{\'\i}nez-Perales, P{\'e}rez and V{\"a}h{\"a}kangas~\cite{hurri2022} 
from $A_1$ weights to $A_p$ weights.
Self-improving results are discussed in
Canto and P\'{e}rez~\cite{CantoPerez2021},
Franchi, P\'{e}rez and Wheeden~\cite{FranchiPerezWheeden1998}, 
Lerner, Lorist and Ombrosi~\cite{LernerLoristOmbrosi2022}
and P\'{e}rez and Rela~\cite{PerezRela2019}.
For fractional truncation methods, see
Chua~\cite{chua2019},
Dyda, Ihnatsyeva and V\"{a}h\"{a}kangas~\cite{DydaLizavetaVahakangas2016},
Dyda, Lehrb{\"a}ck and V{\"a}h{\"a}kangas~\cite{DydaLehrbackVahakangas2022}
and Maz'ya~\cite{mazyabook}.

Our proof for the fractional Poincar\'e inequality also works when we measure the oscillation against a Radon measure \(\mu\).
Our main result \cref{thm:wfracpoincare} states that
\[
\biggl( \int_Q \lvert f-f_Q \rvert^q \intd\mu \biggr)^\frac{1}{q}
\leq
C
(1-\delta)
\int_Q
\int_Q
\f{|f(x)-f(y)|}{|x-y|^{n+\delta}}
\intd y
\, (\M_\alpha\mu(x))^\frac{1}{q}
\intd x
\]
for $0\leq\delta<1$, $1\leq q \leq \frac{n}{n-\delta}$ and where \(\M_\alpha\mu\) is the fractional maximal function with $\alpha = n-q(n-\delta)$.
This extends~\cite[Theorem~2.10]{hurri2022} to all values $0\leq\delta<1$ and exponents $1\leq q\leq\frac{n}{n-\delta}$.
Weighted classical Poincar\'e inequalities have been studied extensively
starting from the classical result by Meyers and Ziemer~\cite{MeyersZiemer1977}
and generalized to
\begin{equation}
\label{intro:weightpoin}
\biggl( \int_Q \lvert f-f_Q \rvert^q \intd\mu \biggr)^\frac{1}{q}
\leq
C
\int_Q
|\nabla f|
\,
(\M_\alpha\mu)^\frac{1}{q}
\intd x
\end{equation}
for $1\leq q \leq \frac{n}{n-1}$, $\alpha = n-q(n-1)$
by Franchi, P\'{e}rez and Wheeden in \cite{perez2000}.
With \cref{thm:wfracpoincare} we extend their results to the fractional setting and are also able to deduce their original results from ours, see \cref{cor:fracpoin_growthcond,cor:weight-clas-poincare}.
Moreover, in~\cite{perez2000} they show \cref{intro:weightpoin} in two separate ranges of $q$ and their constant \(C\) blows up when $q\to1$. 
In our argument, \(C\) is uniformly bounded in $q$ and depends only on the dimension.
We also give an alternative proof by applying the relative isoperimetric inequality to highlight the differences between the classical and the fractional Poincar\'e inequalities.

It would be a natural question to ask if the weighted fractional or classical Poincar\'e inequality holds for $p>1$ as in \cref{intro:claspoin,intro:fracpoin}.
However, this is not the case.
We construct counterexamples in \cref{sec:counterexample}.
This answers a question regarding the weighted classical Poincar\'e inequality posed in~\cite{perez2000}.

\section{Preliminaries}

Let $\mathbb{N} = \{1,2,\dots\}$ and $\mathbb{N}_0 = \mathbb{N}\cup\{0\}$.
Unless otherwise stated, constants are positive and dependent only on the dimension $n$.
We denote the standard Euclidean norm of a point $x\in\mathbb{R}^n$ by $\lvert x \rvert$.
The Lebesgue measure of a measurable subset  $A$ of $\mathbb{R}^n$ is denoted by $\lm{A}$
and the $s$-dimensional Hausdorff measure is denoted by $\mathcal{H}^{s}(A)$.
The absolute continuity of a measure $\mu$ with respect to another measure $\nu$ is denoted by $\mu\ll\nu$, that is, $\nu(A)=0$ implies $\mu(A)=0$.

Assume that $A \subset \mathbb{R}^n$ is a measurable set with $0<\lm{A}<\infty$ and that $f:A\rightarrow [-\infty, \infty]$ is a measurable function.
The maximal median of $f$ over $A$ is defined by
\[
m_f(A) = \inf
\biggl\{
a \in \mathbb{R} : \lm{\{x \in A: f(x) > a\}} < \frac{1}{2} \lm A
\biggr\} .
\]
The integral average of $f \in L^1(A)$ on $A$
is denoted by
\[
f_A = \avint_{A} f \intd x = \frac{1}{\lm{A}} \int_{A} f \intd x .
\]
We write
\[
\{f>\lambda\} = \{x\in\mathbb{R}^n : f(x)>\lambda \}
\]
for the superlevel set of a function $f:\mathbb{R}^n\to\mathbb{R}$. 
We define $\{f<\lambda\}$ similarly.

A cube $Q\subset\mathbb R^n$ is the product of \(n\) closed intervals of the same length, with sides parallel to the coordinate axes and equally long, that is,
$Q=[a_1,a_1+l] \times\dots\times [a_n,a_n+l]$.
In particular, we always consider a cube to be closed and axes-parallel.
All our results hold for half open cubes as well.
If we additionally assume that the measures in our results are absolutely continuous with respect to the Lebesgue measure then we can also use open cubes.
We denote by \(\sle(Q)=l\) the side length of $Q$.

Let $Q_0 \subset \mathbb R^n$ be a cube. 
For each \(k\in\mathbb N_0\) we denote by \(\mathcal D_k(Q_0)\) the set of dyadic subcubes of \(Q_0\) of generation \(k\).
Particularly, $\mathcal{D}_k(Q_0)$ consists of $2^{kn}$ cubes $Q$
with pairwise disjoint interiors and 
with side length $\sle(Q)=2^{-k}\sle(Q_0)$,
such that $Q_0$ equals the union of all cubes in \(\mathcal D_k\) up to a set of measure zero.
If $k\ge 1$ and $Q\in\mathcal{D}_k(Q_0)$, there exists a unique cube $Q'\in\mathcal{D}_{k-1}(Q_0)$
with $Q\subset Q'$. The cube $Q'$ is called the dyadic parent of $Q$, and $Q$ is a dyadic child of $Q'$.
The set of dyadic subcubes $\mathcal{D}(Q_0)$ of $Q_0$ is defined as
$\mathcal{D}(Q_0)=\bigcup_{k=0}^\infty\mathcal{D}_k(Q_0)$.

The following \lcnamecref{C-Z} is a variant of the classical Calder\'{o}n--Zygmund decomposition for sets.

\begin{lemma}
\label{C-Z}
Let $Q \subset \mathbb R^n$ be a cube and $E \subset \mathbb R^n$ a measurable set.
Assume that
\[
\lm{Q\cap E} \leq \lambda \lm{Q}
\]
holds for some $0<\lambda<1$.
Then there exist countably many pairwise disjoint dyadic cubes $Q_i \in\mathcal D(Q)$, $i\in\mathbb N$, such that
\begin{enumerate}[(i)]
\item $Q \cap E \subset \bigcup_i Q_i$ up to a set of Lebesgue measure zero,
\item $ \lm{Q_i \cap E} > 2^{-n} \lambda \lm{Q_i}$,
\item $ \lm{Q_i \cap E} \leq \lambda \lm{Q_i}$.
\end{enumerate}
If \(E\) is relatively open with respect to \(Q\) then $Q \cap E \subset \bigcup_i Q_i$ holds literally and not only up to a set of measure zero.
The cubes in the collection $\{Q_i\}_{i\in\mathbb N}$ are called the Calder\'{o}n--Zygmund cubes in $Q$ at level~$\lambda$.
\end{lemma}

\begin{proof}

If
\[
\lm{Q\cap E} > 2^{-n} \lambda \lm{Q} ,
\]
we pick $Q$ and observe that $Q$ satisfies the required properties.
Otherwise, if 
\[
\lm{Q\cap E} \leq 2^{-n} \lambda \lm{Q} ,
\]
we decompose $Q$ into dyadic subcubes that satisfy the required properties in the following way.
Start by decomposing
$Q$ into $2^n$ dyadic subcubes $Q_1 \in\mathcal{D}_1(Q)$.
We select those $Q_1$ for which
$\lm{Q_1 \cap E} > 2^{-n} \lambda \lm{Q_1}$ and denote this collection by $\{Q_{1,j}\}_j$.
If $\lm{Q_1 \cap E} \leq 2^{-n} \lambda \lm{Q_1}$,
we subdivide $Q_1$ into $2^n$ dyadic subcubes $Q_2\in\mathcal{D}_2(Q)$ and select $Q_2$ for which
$\lm{Q_2 \cap E} > 2^{-n} \lambda \lm{Q_2}$.
We denote so obtained collection by $\{Q_{2,j}\}_j$.

At the $i$th step, we partition unselected $Q_{i-1}$ into dyadic subcubes $Q_i\in\mathcal{D}_i(Q)$ and select those $Q_i$ for which $\lm{Q_i \cap E} > 2^{-n} \lambda \lm{Q_i}$.
Denote the obtained collection by $\{Q_{i,j}\}_j$.
If $\lm{Q_i \cap E} \leq 2^{-n} \lambda \lm{Q_i}$, we continue the selection process in $Q_i$.
In this manner we obtain a collection 
$\{Q_{i,j}\}_{i,j}$
of pairwise disjoint dyadic subcubes of $Q$.
Reindex $\{Q_i\}_i = \{Q_{i,j}\}_{i,j}$.
We show that $\{Q_i\}_i$ satisfies the required properties.

Let $x\in Q\setminus \bigcup_i Q_i$. There exists a decreasing sequence $\{Q_k\}_k$ of dyadic subcubes of \(Q\) containing $x$ such that $Q_{k+1} \subsetneq Q_k$
and
$\lm{Q_k \cap E} \leq 2^{-n} \lambda \lm{Q_k}$ for every $k\in\mathbb{N}$.
If \(E\) is relatively open then for \(k\) large enough we have \(Q_k\subset E\cap Q\), a contradiction.
If \(E\) is a general measurable set, then we have by the Lebesgue differentiation theorem that $\ind{E}(x) \leq 2^{-n}\lambda$ for almost every $x\in Q\setminus \bigcup_i Q_i$ 
and thus $Q \cap E \subset \bigcup_i Q_i$ up to a set of Lebesgue measure zero.
This proves (i).
Property (ii) holds 
by the definition of $Q_i$.
By the selection process, it holds that $\lm{Q'_i \cap E} \leq 2^{-n} \lambda \lm{Q'_i}$ for every $i\in\mathbb{N}$, where $Q'_i$ is the dyadic parent cube of $Q_i$.
Hence, we have
\[
\lm{Q_i \cap E} \leq \lm{Q'_i \cap E} \leq 2^{-n} \lambda \lm{Q'_i} = \lambda \lm{Q_i}.
\]
This proves (iii).
\end{proof}

Let $\mu$ be a Radon measure.
The fractional maximal function of $\mu$ is defined by
\[
\M_\alpha\mu(x) = \sup_{Q \ni x} \sle(Q)^\alpha \f{\mu(Q)}{\lm{Q}} .
\]
For $\alpha=0$, we have the classical Hardy--Littlewood maximal function $\M=\M_0$.
Let $Q_0\subset\mathbb{R}^n$.
The dyadic local counterpart is defined by
\[
\Md_{\alpha,Q_0}\mu(x)
=
\sup_{\substack{ Q\ni x, \\ Q\in\mathcal D(Q_0) }}
\sle(Q)^\alpha \f{\mu(Q)}{\lm{Q}}
,
\]
where we take the supremum only over the dyadic subcubes of \(Q_0\).

For a measurable set \(E\subset\mathbb{R}^n\) denote by \(\ti E\), \(\tc E\) and \(\tb E\) the topological interior, closure and boundary of \(E\), respectively.
The measure theoretic closure and the measure theoretic boundary of $E$ are defined by
\[
\mc E
=
\biggl\{x:\limsup_{r\rightarrow0}\f{\lm{B(x,r)\cap E}}{r^n}>0\biggr\}
\qquad\text{and}\qquad
\mb E
=
\mc E
\cap
\mc{\mathbb{R}^n\setminus E}
.
\]
The measure theoretic versions are robust against changes with measure zero.
Note that \(\mc E\subset\tc E\) and thus \(\mb E\subset\tb E\).
For a cube, its measure theoretic boundary and its closure agree with the respective topological quantities.

We will need the following relative isoperimetric inequality~\cite[Theorem~5.11]{evansgariepy}.
\begin{lemma}
\label{rel.iso.ineq}
Let $Q \subset \mathbb R^n$ be a cube and $E$ a set of finite perimeter. Then there exists a dimensional constant $C$ such that
\[
\min\bigl\{ \lm{Q \cap E}, \lm{Q \setminus E} \bigr\}^\frac{n-1}{n} \leq C  \sm{Q \cap \mb{E}} .
\]
\end{lemma}

\section{Fractional type isoperimetric inequality}

This section discusses a rougher fractional type isoperimetric inequality, \cref{lem:cla_weakcharf}, which is used later to prove the weighted fractional Poincar\'e inequality.
To prove this fractional isoperimetric inequality, we need first some auxiliary results.

\begin{lemma}
\label{lem_smallregion}
Let $Q_0\subset\mathbb{R}^n$ be a cube, $a\leq\sle(Q_0)/2$ and $0<\varepsilon<\frac{1}{2}$.
Let $Q\subset Q_0$ be a cube with \(\sle(Q) \leq a\f{\sqrt\pi}{2^{n+4}n}\).
Then for any measurable set $E\subset\mathbb{R}^n$ with
\[
\varepsilon
\leq
\f{
\lm{Q \cap E}
}{
\lm{Q}
}
\leq
1-\varepsilon
\]
we have
\[
\lm{Q}
\leq
\f4\varepsilon
\int_{Q}
\biggl|
\ind E(x)
- 
\f{
\lm{A(x) \cap E}
}{
\lm{A(x)}
}
\biggr|
\intd x ,
\]
where 
$A(x) = Q_0 \cap B(x,a)\setminus B(x,a/2)$.
\end{lemma}

\begin{proof}

Denote the center of \(Q\) by \(x_0\).
Let \(x\in Q\).
Then we have
\begin{equation}
\label{eq:boundondist}
|x-x_0|
\leq
\f{\sqrt n}2\sle(Q)
\leq
a\f{\sqrt\pi}{2^{n+5}\sqrt n}
.
\end{equation}
Our first step is to show that \cref{eq:boundondist} implies
\begin{equation}
\label{eq:average_difference}
\biggl|
\f{
\lm{A(x) \cap E}
}{
\lm{A(x)}
}
-
\f{
\lm{A(x_0) \cap E}
}{
\lm{A(x_0)}
}
\biggr|
\leq
\f14
.
\end{equation}
Denote by \(\sigma_n\) the \(n\)-dimensional Lebesgue measure of the unit ball in \(n\) dimensions.
Then
\begin{align*}
\bigl|
\lm{A(x) \cap E}
-
\lm{A(x_0) \cap E}
\bigr|
&=
\bigl|
\lm{A(x) \cap E \setminus A(x_0)}
-
\lm{A(x_0) \cap E \setminus A(x)}
\bigr|
\\
&\leq
\max\bigl\{
\lm{A(x_0)\setminus A(x)}
,
\lm{A(x)\setminus A(x_0)}
\bigr\}
\\
&\leq
(a^{n-1}\sigma_{n-1}+(a/2)^{n-1}\sigma_{n-1})
|x-x_0|
\\
&=
(1+2^{-n+1})a^{n-1}\sigma_{n-1}
|x-x_0|,
\end{align*}
\begin{figure}%
\includegraphics{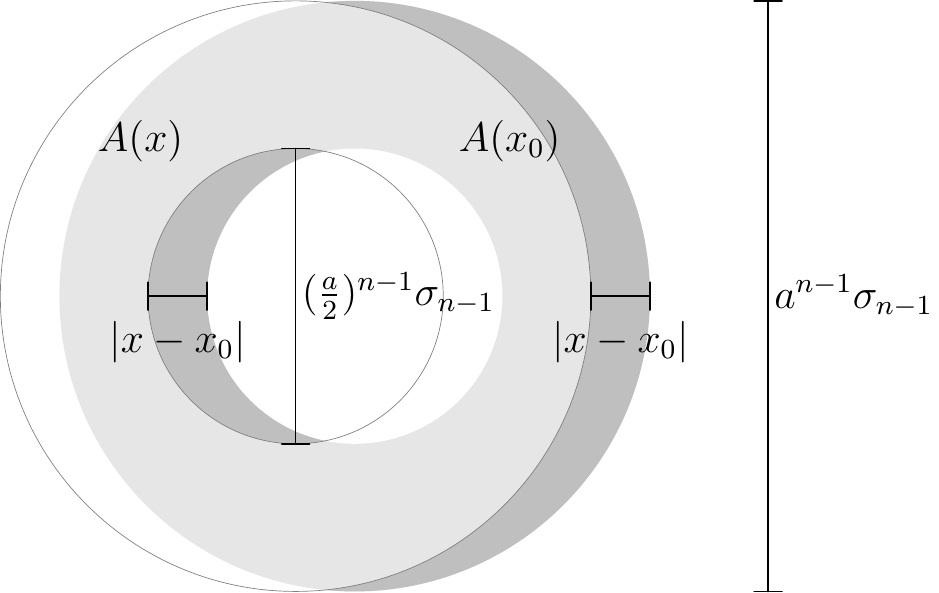}%
\caption{Difference of two shifted annuli.}%
\label{fig:adifference}%
\end{figure}%
where the second inequality follows from the fact that we can estimate the difference of shifted annuli by two differences of shifted balls as illustrated in \cref{fig:adifference}.
This implies
\begin{equation}
\label{eq:multiplied_difference}
\begin{split}
&\bigl|
\lm{A(x) \cap E}
\lm{A(x_0)}
-
\lm{A(x_0) \cap E}
\lm{A(x)}
\bigr|
\\
&\qquad\leq
\bigl|
\lm{A(x) \cap E}
\lm{A(x_0)}
-
\lm{A(x_0) \cap E}
\lm{A(x_0)}
\bigr|
\\
&\qquad\quad
+
\bigl|
\lm{A(x_0) \cap E}
\lm{A(x_0)}
-
\lm{A(x_0) \cap E}
\lm{A(x)}
\bigr|
\\
&\qquad=
\bigl|
\lm{A(x) \cap E}
-
\lm{A(x_0) \cap E}
\bigr|
\lm{A(x_0)}
\\
&\qquad\quad
+
\bigl|
\lm{A(x_0)}
-
\lm{A(x)}
\bigr|
\lm{A(x_0) \cap E}
\\
&\qquad\leq
2
(1+2^{-n+1})a^{n-1}\sigma_{n-1}
|x-x_0|
\lm{A(x_0)}
.
\end{split}
\end{equation}
By the formula \(\sigma_n=\pi^{\f n2}/\Gamma(\frac{n}{2}+1)\) and~\cite{wendel1948}, we have
\[
\f{\sigma_{n-1}}{\sigma_n}
\leq
\sqrt{
\f{n+1}{2\pi}
}
.
\]
The inequality
\[
\f
{(1+2^{-n+1})\sqrt{n+1}}
{(1-2^{-n})\sqrt n}
\leq
4\sqrt 2
\]
clearly holds for \(n=1\), and thus for all \(n\in\mathbb N\), as the left-hand side is decreasing in \(n\).
Combining the two previous inequalities with \cref{eq:boundondist}
and
\begin{equation}
\label{eq:Avolume2}
\lm{A(x)}
\geq
\f{1-2^{-n}}{2^n} 
\sigma_n a^n ,
\end{equation}
we obtain
\[
(1+2^{-n+1})a^{n-1}\sigma_{n-1}
|x-x_0|
\leq
\f18\lm{A(x)}
.
\]
Thus, \cref{eq:multiplied_difference} implies
\begin{align*}
\bigl|
\lm{A(x) \cap E}
\lm{A(x_0)}
-
\lm{A(x_0) \cap E}
\lm{A(x)}
\bigr|
\leq
\f14
\lm{A(x)}
\lm{A(x_0)}
.
\end{align*}
Dividing the previous inequality by \(\lm{A(x)}\lm{A(x_0)}\), we conclude \cref{eq:average_difference}.

If
\[
\f{
\lm{A(x_0) \cap E}
}{
\lm{A(x_0)}
}
\geq
\f12 ,
\]
then it holds that
\begin{align*}
\f{
\lm{A(x) \cap E}
}{
\lm{A(x)}
}
\geq
\f{
\lm{A(x_0) \cap E}
}{
\lm{A(x_0)}
}
-
\biggl|
\f{
\lm{A(x) \cap E}
}{
\lm{A(x)}
}
-
\f{
\lm{A(x_0) \cap E}
}{
\lm{A(x_0)}
}
\biggr|
\geq
\f12
-
\f14
=
\f14 .
\end{align*}
On the other hand, if
\[
\f{
\lm{A(x_0) \cap E}
}{
\lm{A(x_0)}
}
<
\f12 ,
\]
then
\begin{align*}
\biggl|
1-
\f{
\lm{A(x) \cap E}
}{
\lm{A(x)}
}
\biggr|
&\geq
1
-
\f{
\lm{A(x_0) \cap E}
}{
\lm{A(x_0)}
}
-
\biggl|
\f{
\lm{A(x) \cap E}
}{
\lm{A(x)}
}
-
\f{
\lm{A(x_0) \cap E}
}{
\lm{A(x_0)}
}
\biggr|
\geq
1
-
\f12
-
\f14
=
\f14 .
\end{align*}
As a consequence there is \(i\in\{0,1\}\) such that
\[
\biggl|
i-
\f{
\lm{A(x) \cap E}
}{
\lm{A(x)}
}
\biggr|
\geq
\f14
\]
for every \(x\in Q\).
Denote \(F=E\) if \(i=1\) and \(F=\mathbb{R}^n\setminus E\) if \(i=0\).
By the assumption, we have
\[
\lm{Q \cap F} 
\geq 
\varepsilon 
\lm{Q} .
\]
Then we can conclude that
\begin{align*}
\lm{Q}
&\leq
\f1\varepsilon
\int_{Q \cap F}
1
\intd x
\leq
\f4\varepsilon
\int_{Q \cap F}
\biggl|
i
-
\f{
\lm{A(x) \cap E}
}{
\lm{A(x)}
}
\biggr|
\intd x
\leq
\f4\varepsilon
\int_{Q}
\biggl|
\ind E(x)
-
\f{
\lm{A(x) \cap E}
}{
\lm{A(x)}
}
\biggr|
\intd x
.
\end{align*}
The proof is complete.
\end{proof}

The next \lcnamecref{lem_onescale} is a decomposition of a set near its boundary into cubes.

\begin{lemma}
\label{lem_onescale}
Let $Q_0 \subset \mathbb R^n$ be a cube and $E \subset \mathbb R^n$ a measurable set
such that
\[
\f1{2^{n+1}}
\leq
\f{
\lm{Q_0 \cap E}
}{
\lm{Q_0}
}
\leq
\f12 .
\]
Let \(k\in\mathbb N_0\).
Then there exist a dimensional constant $C$ and pairwise disjoint cubes $Q_1,\ldots,Q_N \subset Q_0$ such that $\sle(Q_i)=2^{-k}\sle(Q_0)$ and
\[
\f1{2^{n+2}}
\leq
\f{
\lm{Q_i \cap E}
}{
\lm{Q_i}
}
\leq
\f34
\]
for every \(i=1,\ldots,N\)
and
\[
2^{-k}
\lm{Q_0}
\leq
C
\sum_{i=1}^N
\lm{Q_i}
.
\]
\end{lemma}

\begin{proof}

Recall that \(\mathcal D_k(Q_0)\) is the set of dyadic subcubes of \(Q_0\) of generation \(k\).
In particular, \(\mathcal D_k(Q_0)\) consists of \(2^{nk}\) many pairwise disjoint cubes with side length \(2^{-k}\sle(Q_0)\) which decompose \(Q_0\).
Denote by $\mathcal{Q}$ the collection of those dyadic subcubes \(Q\in\mathcal D_k(Q_0)\) with
\[
\lm{Q\cap E}
\geq
\f1{2^{n+2}}
\lm Q
\]
and let
$A = \bigcup_{Q\in \mathcal{Q}} Q$.
We have
\begin{align*}
\lm{Q_0\cap E\setminus A}
=
\sum_{Q\in \mathcal D_k(Q_0)\setminus \mathcal{Q}} 
\lm{Q \cap E}
\leq
\f1{2^{n+2}}
\sum_{Q\in \mathcal D_k(Q_0)\setminus \mathcal{Q}}
\lm{Q}
\leq
\f1{2^{n+2}}
\lm{Q_0} ,
\end{align*}
and thus
\begin{equation}
\label{est:A_lowerbound}
\begin{split}
\lm{A}
&\geq
\lm{A \cap E}
=
\lm{Q_0  \cap E}
-
\lm{Q_0\cap E \setminus A \cap E}
\\
&\geq
\f1{2^{n+1}} \lm{Q_0} 
- 
\f1{2^{n+2}} \lm{Q_0}
=
\f1{2^{n+1}}
\lm{Q_0} .
\end{split}
\end{equation}
Denote
\[
\mathcal{A} 
= 
\bigl\{
Q \in \mathcal{Q}: \sm{\mathring{Q}_0 \cap \mb{A} \cap \partial Q } \geq\sle(Q)^{n-1}
\bigr\} 
,
\]
where \(\mathring{Q}_0\) is the interior of \(Q_0\),
so \(\mathcal A\) is the set of those cubes in $\mathcal{Q}$ 
that have at least one of their faces contained in $\mathring{Q}_0\cap\mb{A}$.
Note that $\mathring{Q}_0 \cap \mb{A} \subset \bigcup_{Q\in\mathcal{A}} \partial Q $.
For every cube $Q\in\mathcal{A}$, there exists a neighbouring dyadic cube 
$P\in\mathcal D_k(Q_0)\setminus \mathcal{Q}$.
Thus, 
the cube 
\(\widetilde Q_\lambda=(1-\lambda)Q+\lambda P\)
with side length \(2^{-k}\sle(Q_0)\)
is contained in \(Q\cup P\) for every \(0\leq\lambda\leq1\).
By the definition of $\mathcal{Q}$ and $\mathcal D_k(Q_0)\setminus \mathcal{Q}$, there exists \(0\leq\lambda\leq1\) such that
\[
\lm{\widetilde Q_\lambda\cap E}
=
\f1{2^{n+2}}
\lm{\widetilde Q_\lambda}
.
\]
We denote \(\widetilde Q=\widetilde Q_\lambda\) for this \(\lambda\).
The collection 
$\{\widetilde{Q}:Q\in\mathcal{A}\}$
is not necessarily disjoint. 
Observe that every cube in \(\mathcal D_k\) has $2n$ faces and thus at most \(2n\) neighbouring cubes in \(\mathcal D_k\).
Hence, for every \(x\in Q_0\) there are at most \(2n\) many cubes $\widetilde{Q}$ with $x\in\widetilde{Q}$.
Let $\lvert\mathcal{A}\rvert$ denote the number of cubes in $\mathcal{A}$.
Thus, we may extract a maximal disjoint subcollection 
$\widetilde{\mathcal{A}} \subset 
\{\widetilde{Q}:Q\in\mathcal{A}\}$
such that
$\lvert\mathcal A\rvert \leq 2n \lvert \widetilde{\mathcal A}\rvert$.

If $\lm{A} \leq \f34\lm{Q_0}$, then by  
\cref{est:A_lowerbound,rel.iso.ineq}, we have
\begin{align*}
\biggl(
\f
{\lm{Q_0}}
{2^{n+1}}
\biggr)^\f{n-1}{n}
&=
\min\biggl\{
\f
{\lm{Q_0}}
{2^{n+1}}
,
\f
{\lm{Q_0}}
4
\biggr\}^\f{n-1}{n}
\\
&\leq
\min\bigl\{
\lm{Q_0 \cap A},
\lm{Q_0 \setminus A}
\bigr\}^\f{n-1}{n}
\\
&\leq 
C_1
\sm{Q_0 \cap \mb{A} }
\leq
C_1
\sum_{Q\in \mathcal{A}} 
\sm{\partial Q}
\\
&= 
C_1
\lvert 
\mathcal A 
\rvert
2^{-k(n-1)}\sm{\partial Q_0}
\\
&\leq 
C_1
4n^2
\lvert 
\widetilde{\mathcal A} 
\rvert
2^{-k(n-1)}\sle(Q_0)^{n-1}
\\
&=
\f{C_1 4n^2 2^{k}}{\sle(Q_0)}
\sum_{Q\in\widetilde{\mathcal A}}
\lm Q
,
\end{align*}
where $C_1$ is the constant in \cref{rel.iso.ineq}.
Thus, it holds that
\[
2^{-k} \lm{Q_0}
\leq
C
\sum_{Q\in\widetilde{\mathcal A}}
\lm Q ,
\]
where $C = 2^{n+2} n^2 C_1$.
Hence, the cubes 
$\{Q_1,\ldots,Q_N\}=\widetilde{A}$
satisfy the conclusion of the \lcnamecref{lem_onescale}.

It remains to consider the case $\lm{A} > \f34\lm{Q_0}$.
We define
\[
\{Q_1,\ldots,Q_N\}
=
\biggl\{
Q\in\mathcal Q:
\lm{Q\cap E} \leq \f{3}{4} \lm{Q}
\biggr\}
=
\biggl\{
Q\in\mathcal D_k:
\f1{2^{n+2}}
\leq
\f
{\lm{Q\cap E}}
{\lm{Q}}
\leq
\f{3}{4}
\biggr\} .
\]
Then we have
\begin{align*}
\sum_{
Q \in \mathcal Q
\setminus
\{Q_1,\ldots,Q_N\}
}
\lm{Q}
\leq
\f{4}{3}
\sum_{
Q \in \mathcal Q
\setminus
\{Q_1,\ldots,Q_N\}
}
\lm{Q \cap E}
\leq
\f{4}{3}
\lm{Q_0 \cap E}
\leq
\f{2}{3}
\lm{Q_0} .
\end{align*}
We conclude that
\begin{align*}
2^{-k}\lm{Q_0}
&\leq
\lm{Q_0}
=
12
\biggl(
\f{3}{4}
-
\f{2}{3}
\biggr)
\lm{Q_0}
\\
&\leq
12 \sum_{Q \in \mathcal{Q}}
\lm{Q}
-
12 \sum_{Q \in \mathcal{Q}\setminus\{Q_1,\ldots,Q_N\}}
\lm{Q}
\\
&=
12 \sum_{i=1}^N
\lm{Q_i} .
\end{align*}
This completes the proof.
\end{proof}

We are ready to prove the following rougher version of the isoperimetric inequality.
Observe that the difference compared to the relative isoperimetric inequality (\cref{rel.iso.ineq}) 
is that the right-hand side in \cref{lem:cla_weakcharf} measures the area around the boundary by annuli of certain size.

\begin{lemma}
\label{lem:cla_weakcharf}
Let $Q \subset \mathbb R^n$ be a cube, \(E\subset \mathbb{R}^n\) a measurable set, \(k\in\mathbb N\) and \(s\geq0\) such that
\[
\frac{1}{2^{(k+s)n}}
\leq
\f
{\lm{Q\cap E}}
{\lm Q}
\leq
\f12
.
\]
Then there exists a dimensional constant 
$C$ 
such that
\[
\biggl(
\f
{\lm{Q\cap E}}
{\lm Q}
\biggr)^{\f{n-1}n}
\leq C
2^{k+s}
\avint_{Q}
\avint_{Q\cap B(x,2^{-k}\sle(Q))\setminus B(x,2^{-k-1}\sle(Q))}
|\ind E(x)-\ind E(y)|
\intd y
\intd x
.
\]
\end{lemma}

\begin{proof}

Both sides of the claim are invariant under the dilation of \(Q\) and \(E\) by the same factor.
Hence, it suffices to consider the case $\sle(Q)=1$.

By the assumption of the \lcnamecref{lem:cla_weakcharf}, we may apply \cref{C-Z} for $E$ on $Q$ at level $\frac12$.
Thus, we obtain a collection \(\{Q_i\}_i\) of Calder\'{o}n--Zygmund cubes such that
$Q\cap E \subset \bigcup_i Q_i$ up to a set of Lebesgue measure zero and
\[
\frac{1}{2^{n+1}} < \frac{\lm{Q_i\cap E}}{\lm{Q_i}} \leq \frac{1}{2} 
\]
for every $i\in\mathbb{N}$.
Note that 
$\sle(Q_i)=2^{-M_i}$
for some $M_i\in\mathbb{N}_0$.
Denote by \(K\in\mathbb{N}\) the smallest integer with 
\(2^K\geq\f{2^{n+4}n}{\sqrt\pi}\).
We apply \cref{lem_onescale}
with $\max\{k+K-M_i, 0\}$ for $E$ on each $Q_i$.
Then for every $i\in\mathbb{N}$
we obtain a collection 
$\{Q_{i,1},\ldots,Q_{i,N_i}\}$
of pairwise disjoint subcubes 
with
\[
\sle(Q_{i,j})=2^{-\max\{k+K-M_i, 0\}} \sle(Q_i) =\min\{2^{-k-K},\sle(Q_i)\}
\]
such that
\[
\frac{1}{2^{n+2}}
\leq
\f{
\lm{Q_{i,j} \cap E}
}{
\lm{Q_{i,j}}
}
\leq
\f34
\]
for every \(j=1,\ldots,N_i\)
and
\[
\min\bigl\{
2^{-k-K}\lm{Q_i}^{\f{n-1}n},\lm{Q_i}
\bigr\}
=
2^{-\max\{k+K-M_i, 0\}}
\lm{Q_i}
\leq
C_1
\sum_{j=1}^{N_i}
\lm{Q_{i,j}}
,
\]
where 
$C_1$ is the constant in \cref{lem_onescale}.
By the properties of $Q_i$, the assumption $2^{-k-c} \leq \lm{Q\cap E}^\f1n$ 
and the previous inequality, we get
\begin{align*}
\lm{Q\cap E}^{\f{n-1}n}
&\leq
\lm{Q \cap E}^{-\frac{1}{n}}
\sum_{i}
\lm{Q_i}
\\
&=
\min\bigl\{
\lm{Q \cap E}^{-\frac{1}{n}},2^{k+s}
\bigr\}
\sum_{i}
\lm{Q_i}
\\
&\leq
\sum_{i}
\min\bigl\{
\lm{Q_i \cap E}^{-\frac{1}{n}},2^{k+s}
\bigr\}
\lm{Q_i}
\\
&\leq
\sum_{i}
\min\bigl\{
2^{1+\frac{1}{n}}
\lm{Q_i}^{-\frac{1}{n}}
,2^{k+s}
\bigr\}
\lm{Q_i}
\\
&=
2^{k+K}
\sum_{i}
\min\bigl\{
2^{1+\frac{1}{n}-k-K}
\lm{Q_i}^{\frac{n-1}{n}}
,2^{s-K} \lm{Q_i}
\bigr\}
\\
&\leq
2^{k+s+K+1+\frac{1}{n}}
\sum_{i}
\min\bigl\{
2^{-k-K} 
\lm{Q_i}^{\frac{n-1}{n}}
,\lm{Q_i}
\bigr\}
\\
&\leq
2^{k+s+K+1+\frac{1}{n}}
C_1
\sum_{i,j}
\lm{Q_{i,j}}
.
\end{align*}
Using \cref{lem_smallregion} with $\varepsilon=1/2^{n+2}$, 
we obtain
\[
\lm{Q_{i,j}}
\leq
2^{n+4}
\int_{Q_{i,j}}
\biggl|
\ind E(x)
- 
\f{
\lm{A(x) \cap E}
}{
\lm{A(x)}
}
\biggr|
\intd x
\]
for every $i,j\in\mathbb{N}$,
where 
$A(x) = Q \cap B(x,2^{-k})\setminus B(x,2^{-k-1})$.
Thus, we may estimate
\begin{align*}
\sum_{i,j}\lm{Q_{i,j}}
&\leq
2^{n+4}
\sum_{i,j}
\int_{Q_{i,j}}
\biggl|
\ind E(x)
- 
\f{
\lm{A(x) \cap E}
}{
\lm{A(x)}
}
\biggr|
\intd x 
\\
&=
2^{n+4}
\sum_{i,j}
\int_{Q_{i,j}}
\biggl|
\ind E(x)- \avint_{A(x)} \ind E(y) \intd y
\biggr|
\intd x
\\
&\leq
2^{n+4}
\sum_{i,j}
\int_{Q_{i,j}}
\avint_{A(x)}
|\ind E(x)-\ind E(y)|
\intd y
\intd x
\\
&\leq
2^{n+4}
\int_{Q}
\avint_{A(x)}
|\ind E(x)-\ind E(y)|
\intd y
\intd x .
\end{align*}
Combining the obtained estimates,
we conclude that
\[
\lm{Q\cap E}^\frac{n-1}{n}
\leq
C
2^{k+s}
\int_{Q}
\avint_{Q\cap B(x,2^{-k})\setminus B(x,2^{-k-1})}
|\ind E(x)-\ind E(y)|
\intd y
\intd x ,
\]
where $C=\f{n}{\sqrt\pi} 2^{2n+11}C_1$.
This completes the proof.
\end{proof}

\begin{figure}
\includegraphics{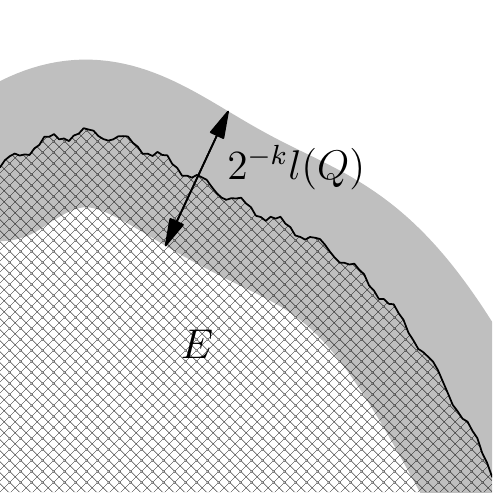}
\caption{For a very regular set the inner integral in \cref{lem:cla_weakcharf} behaves approximately like the characteristic function of a neighborhood of the boundary.}
\label{fig_fractisoperimetric}
\end{figure}

\begin{remark}
\label{rem:cla_weakcharf_rhs}
As mentioned in the introduction, the fractional Poincar\'e inequality \cref{intro:fracpoin} is an improvement of the classical Poincar\'e inequality \cref{intro:claspoin} in the sense that the fractional integral of \(f\) on right hand side of \cref{intro:fracpoin} can be bounded by the integral of the gradient on the right hand side of \cref{intro:claspoin}.
Likewise, \cref{lem:cla_weakcharf} is an improvement of the relative isoperimetric inequality (\cref{rel.iso.ineq}), as
\begin{equation}
\label{eq:cla_weakcharf_rhs}
2^k
\avint_{Q}
\avint_{Q\cap B(x,2^{-k}\sle(Q))\setminus B(x,2^{-k-1}\sle(Q))}
|\ind E(x)-\ind E(y)|
\intd y
\intd x
\leq
C
\f{\sm{Q\cap\mb E}}{\lm Q^{\f{n-1}n}}
\end{equation}
for some dimensional constant \(C\).
Here we do not need to assume any bound on \(\lm{Q\cap E}/\lm Q\).
Note that \cref{lem:cla_weakcharf} still holds if we integrate over \(Q\cap B(x,2^{-k}\sle(Q))\) instead of \(Q\cap B(x,2^{-k}\sle(Q))\setminus B(x,2^{-k-1}\sle(Q))\),
and so does \cref{eq:cla_weakcharf_rhs}.

Define the following averaged out version of the inner integral by
\[
f(z)
=
\f1{\lm{B(0,2^{-k}\sle(Q))}}
\int_{Q\cap B(z,2^{-k}\sle(Q))}
\avint_{Q\cap B(x,2^{-k}\sle(Q))\setminus B(x,2^{-k-1}\sle(Q))}
|\ind E(x)-\ind E(y)|
\intd y
\intd x
.
\]
Then
\[
\int_{\mathbb{R}^n} f(z) \intd z
=
\int_{Q}
\avint_{Q\cap B(x,2^{-k}\sle(Q))\setminus B(x,2^{-k-1}\sle(Q))}
|\ind E(x)-\ind E(y)|
\intd y
\intd x
.
\]
If the boundary of \(E\) is very regular then \(f\) behaves roughly like the characteristic function of the \(2^{-k}\sle(Q)\)-neighborhood of \(Q\cap\mb E\) as shown in \cref{fig_fractisoperimetric},
and its integral evaluates to approximately \(2^{-k}\sle(Q)\sm{Q\cap\mb E}\).
This means the two sides in \cref{eq:cla_weakcharf_rhs} are comparable and \cref{lem:cla_weakcharf} becomes the classical relative isoperimetric inequality.
If the boundary of \(E\) is rougher then \(f\) instead resembles the characteristic function of the neighborhood of a straightened out boundary of \(E\).
Morally this means that \cref{eq:cla_weakcharf_rhs} represents the removal of small wiggles in the boundary of \(E\)
and \cref{lem:cla_weakcharf} holds by the relative isoperimetric inequality since we can replace \(E\) on the left hand side by a straightened out set with similar volume.

For a formal proof of \cref{eq:cla_weakcharf_rhs} define
\[
E_i
=
\{z\in \mathbb{R}^n: f(z)\geq 2^{-i}\}
\setminus\bigcup_{z\in E_1\cup\ldots\cup E_{i-1}}B(z,2^{-k+2}\sle(Q))
\]
recursively for all \(i\in\mathbb{N}\).
By the Vitali covering theorem, for each \(i\) there exists a collection \(\B_i\) of pairwise disjoint balls \(B(z,2^{-k+1}\sle(Q))\) with \(z\in E_i\) such that
\[
\bigcup_{z\in E_i}B(z,2^{-k+2}\sle(Q))
\subset
\bigcup_{B\in \B_i}8B
.
\]
Furthermore, by the definition of \(E_j\) for any \(j>i\) the balls in \(\B_j\) do not intersect the balls in \(\B_i\).
For any \(z\in E_i\) we have
\begin{align*}
2^{-i}
\leq
f(z)
&\leq
\f{2^{n+1}}{\lm{B(0,2^{-k}\sle(Q))}^2}
\int_{Q\cap B(z,2^{-k+1}\sle(Q))}
\int_{Q\cap B(z,2^{-k+1}\sle(Q))}
|\ind E(x)-\ind E(y)|
\intd y
\intd x
\\
&=
\f{2^{n+2}}{\lm{B(0,2^{-k}\sle(Q))}^2}
\lm{Q\cap B(z,2^{-k+1}\sle(Q))\cap E}
\lm{Q\cap B(z,2^{-k+1}\sle(Q))\setminus E}
\\
&\leq
\f{2^{3n+2}
\min\bigl\{
\lm{Q\cap B(z,2^{-k+1}\sle(Q))\cap E}
,
\lm{Q\cap B(z,2^{-k+1}\sle(Q))\setminus E}
\bigr\}
}{
\lm{B(0,2^{-k+1}\sle(Q))}
}
,
\end{align*}
and thus by the relative isoperimetric inequality (\cref{rel.iso.ineq}) for \(Q\cap B(z,2^{-k}\sle(Q))\) 
there exists a constant $C_1$ such that
\begin{align*}
&
2^{-i}
\lm{B(z,2^{-k+1}\sle(Q))}
\\
&\qquad\leq2^{3n+2}
\min\bigl\{
\lm{
Q\cap B(z,2^{-k+1}\sle(Q))\cap E
}
,
\lm{
Q\cap B(z,2^{-k+1}\sle(Q))\setminus E
}
\bigr\}
\\
&\qquad\leq 2^{3n+3} \sigma_n^{\f1n}
2^{-k}\sle(Q)
\min\bigl\{
\lm{
Q\cap B(z,2^{-k+1}\sle(Q))\cap E
}
,
\lm{
Q\cap B(z,2^{-k+1}\sle(Q))\setminus E
}
\bigr\}^{\f{n-1}n}
\\
&\qquad\leq 2^{3n+3} \sigma_n^{\f1n} C_1
2^{-k}\sle(Q)
\sm{
Q\cap B(z,2^{-k+1}\sle(Q))\cap\mb E
}
.
\end{align*}
We can conclude that
\begin{align*}
\int_{\mathbb{R}^n} f(z)\intd z
&\leq
\sum_{i=1}^{\infty}
2^{-i+1}
\lmb{\bigcup_{z\in E_i}B(z,2^{-k+2}\sle(Q))}
\\
&\leq8^n
\sum_{i=1}^{\infty}
2^{-i+1}
\sum_{B\in\B_i}\lm B
\\
&\leq 2^{6n+4} \sigma_n^{\f1n} C_1
2^{-k}\sle(Q)
\sum_{i=1}^{\infty}
\sum_{B\in\B_i}
\sm{Q\cap B(z,2^{-k+1}\sle(Q))\cap\mb E}
\\
&\leq 2^{6n+4} \sigma_n^{\f1n} C_1
2^{-k}\sle(Q)
\sm{Q\cap\mb E}
,
\end{align*}
finishing the proof of \cref{eq:cla_weakcharf_rhs}.
\end{remark}

\section{Weighted fractional $(q,1)$-Poincar\'e inequality} 

In this section, we prove our main result \cref{thm:wfracpoincare},
the weighted fractional Poincar\'e inequality in the case $p=1$.
This improves Theorem~2.10 in \cite{hurri2022}.
Observe that by choosing $\mu = \mathcal{L}$, we obtain the non-weighted fractional Poincar\'e inequality.
Recall that \(\Md_{\alpha,Q}\mu\) is the local fractional dyadic maximal function.
Since it is pointwise bounded by the fractional maximal function, \cref{thm:wfracpoincare} also holds with \(\M_{\alpha}\mu\) in place of \(\Md_{\alpha,Q}\mu\).

\begin{theorem}
\label{thm:wfracpoincare}
Let $0\leq\delta<1$, $1\leq q \leq \frac{n}{n-\delta}$, $\alpha = n-q(n-\delta)$,
$f \in L^1_{\loc}(\mathbb R^n)$ and let
$\mu$ be a Radon measure with $\mu\ll\mathcal{L}$.
Then there exists a dimensional constant $C$ such that
\[
\biggl( \int_Q \lvert f-f_Q \rvert^q \intd\mu \biggr)^\frac{1}{q}
\leq
C
(1-\delta)
\int_Q
\int_Q
\f{|f(x)-f(y)|}{|x-y|^{n+\delta}}
\intd y
\, (\Md_{\alpha,Q}\mu(x))^\frac{1}{q}
\intd x
\]
for every cube $Q \subset \mathbb R^n$. 

Alternatively, we can assume that $\mu$ is a general Radon measure and the claim holds for any continuous function $f$. 
\end{theorem}

Note that the conditions on the parameters in \cref{thm:wfracpoincare} can also be written as \(0\leq\alpha\leq\delta<1\), \(q=\f{n-\alpha}{n-\delta}\).

\begin{remark}
The proof of the theorem given next is for the Lorentz norm, namely the $L^q$ norm in the left of the claim of the theorem can be replaced by the $\|\cdot\|_{L^{q,1}(\mu)}$ norm, namely 
$$
||f-f_Q||_{L^{q,1}(\mu)}
=q  \int_0^\infty \mu(\Omega_\lambda)^\frac{1}{q} \intd\lambda,
$$
where  $\Omega_\lambda = \{x\in Q: \lvert f-f_Q \rvert > \lambda \}$.
 
\end{remark}

\begin{proof}[Proof of \cref{thm:wfracpoincare}]
Fix $Q \subset \mathbb R^n$ and denote $\Omega_\lambda = \{x\in Q: \lvert f-f_Q \rvert > \lambda \}$.
It holds that
\[
\lambda^{q-1} \mu(\Omega_\lambda)^\frac{q-1}{q} \leq \biggl( \int_0^\lambda \mu(\Omega_t)^\frac{1}{q} \intd t \biggr)^{q-1} ,
\]
since $\Omega_\lambda \subset \Omega_t$ for $0<t\leq \lambda$.
By Cavalieri's principle, this implies
\begin{equation}
\label{eq_cavalieriq}
\begin{split}
\biggl( \int_Q \lvert f-f_Q \rvert^q \intd\mu \biggr)^\frac{1}{q} &= \biggl( q \int_0^\infty \lambda^{q-1} \mu(Q \cap \{ \lvert f-f_Q \rvert > \lambda \} ) \intd\lambda  \biggr)^\frac{1}{q} \\
&= \biggl( q \int_0^\infty \lambda^{q-1} \mu(\Omega_\lambda)^\frac{q-1}{q} \mu(\Omega_\lambda)^\frac{1}{q} \intd\lambda  \biggr)^\frac{1}{q} \\
&\leq \Biggl( q \int_0^\infty \biggl( \int_0^\lambda \mu(\Omega_t)^\frac{1}{q} \intd t \biggr)^{q-1} \mu(\Omega_\lambda)^\frac{1}{q} \intd\lambda  \Biggr)^\frac{1}{q} \\
&\leq q^\frac{1}{q} \biggl( \int_0^\infty \mu(\Omega_t)^\frac{1}{q} \intd t \biggr)^\frac{q-1}{q} \biggl(  \int_0^\infty \mu(\Omega_\lambda)^\frac{1}{q} \intd\lambda  \biggr)^\frac{1}{q} \\
&\leq 2 \int_0^\infty \mu(\Omega_\lambda)^\frac{1}{q} \intd\lambda
\\
&=
2\int_{-\infty}^{f_Q} \mu(Q \cap \{f < \lambda \})^\frac{1}{q} \intd\lambda
+
2\int_{f_Q}^\infty \mu(Q \cap \{f > \lambda \})^\frac{1}{q} \intd\lambda
.
\end{split}
\end{equation}
The previous two terms swap when replacing \(f\) by \(-f\).
Thus it suffices to bound the second term.
We split it into two parts
\begin{equation}
\label{split-fractional}
\begin{split}
\int_{f_Q}^\infty \mu(Q \cap \{f > \lambda \})^\frac{1}{q}\intd\lambda
&=
\int_{f_Q}^{\max\{m_f,f_Q\}}
\mu(Q \cap \{f > \lambda \})^\frac{1}{q}
\intd\lambda \\
& \qquad
+
\int_{\max\{m_f,f_Q\}}^\infty
\mu(Q \cap \{f > \lambda \})^\frac{1}{q} \intd\lambda .
\end{split}
\end{equation}
We abbreviate the maximal median of $f$ over $Q$ by \(m_f=m_f(Q)\).

For the first term in~\cref{split-fractional} it suffices to consider \(f_Q<m_f\).
Recall that \( \lm{Q\cap\{f>\lambda\}} \geq \lm Q/2 \) for $\lambda<m_f$.
By the definition of \(f_Q\), it holds that
\[
\int_{f_Q}^\infty
\lm{Q\cap\{f>\lambda\}}
\intd\lambda
=
\int_{-\infty}^{f_Q}
\lm{Q\cap\{f<\lambda\}}
\intd\lambda .
\]
Using these facts, we get
\begin{align*}
\int_{f_Q}^{\max\{m_f,f_Q\}}
\mu(Q \cap\{f>\lambda\})^\frac{1}{q}
\intd\lambda
&
\leq
(m_f-f_Q)
\mu(Q)^\frac{1}{q}
\\
&\leq2
\frac{
\mu(Q)^\frac{1}{q}
}{
\lm{Q}
}
\int_{f_Q}^{m_f}
\lm{Q\cap\{f>\lambda\}}
\intd\lambda
\\
&\leq2
\frac{
\mu(Q)^\frac{1}{q}
}{
\lm{Q}
}
\int_{f_Q}^{\infty}
\lm{Q\cap\{f>\lambda\}}
\intd\lambda
\\
&=2
\frac{
\mu(Q)^\frac{1}{q}
}{
\lm{Q}
}
\int_{-\infty}^{f_Q}
\lm{Q\cap\{f<\lambda\}}
\intd\lambda 
\\
&\leq
2^\frac{n-\delta}{n}
\frac{
\mu(Q)^\frac{1}{q}
}{
\lm{Q}^\frac{n-\delta}{n}
}
\int_{-\infty}^{m_f}
\lm{Q\cap\{f<\lambda\}}^{\frac{n-\delta}n}
\intd\lambda ,
\end{align*}
where in the last inequality we used 
\(\lm{Q\cap\{f<\lambda\}}\leq\lm Q/2\) for \(\lambda<m_f\).
Denote 
\[
A_k(x) = Q \cap B(x,2^{-k}\sle(Q))\setminus B(x,2^{-k-1}\sle(Q))
\]
for $k\in\mathbb{N}$ 
and
\[
K_\lambda
= 
\Bigl\lceil
\log_2\bigl(
\sle(Q)/\lm{Q\cap\{f<\lambda\}}^\f1n
\bigr)
\Bigr\rceil
\]
for $\lambda<m_f$.
Then we have
\[
\frac{1}{2^{kn}} \leq 
\frac{\lm{Q\cap\{f<\lambda\}}}{\lm{Q}} \leq \frac{1}{2}
\]
for every $k\geq K_\lambda$, $\lambda<m_f$.
Thus for each $k\geq K_\lambda$, $\lambda<m_f$, we may apply \cref{lem:cla_weakcharf} with $s=0$ for $E = \{f<\lambda\}$ on $Q$ to obtain
\[
\lm{Q\cap\{f<\lambda\}}^{\f{n-1}n}
\leq
C_1
\f
{2^k}
{\sle(Q)}
\int_{Q}
\avint_{A_k(x)}
|\ind{\{f<\lambda\}} (x)-\ind{\{f<\lambda\}} (y)|
\intd y
\intd x
.
\]
We multiply both sides of the previous estimate by $2^{-k(1-\delta)}$ and sum over $k\geq K_\lambda$ to get
\[
\sum_{k= K_\lambda}^\infty
2^{-k(1-\delta)}
\lm{Q\cap\{f<\lambda\}}^{\f{n-1}n}
\leq
C_1
\sum_{k= K_\lambda}^\infty
\f
{2^{k\delta}}
{\sle(Q)}
\int_{Q}
\avint_{A_k(x)}
|\ind{\{f<\lambda\}} (x)-\ind{\{f<\lambda\}} (y)|
\intd y
\intd x
.
\]
Furthermore,
\begin{align*}
\sum_{k= K_\lambda}^\infty
2^{-k(1-\delta)} 
&=
\f{
2^{-K_\lambda(1-\delta)}
}{
1-2^{-(1-\delta)}
}
\geq
\f{2^{-(1-\delta)}}{1-2^{\delta-1}}
\f{
\lm{Q\cap\{f<\lambda\}}^{\f{1-\delta}n} 
}{
\sle(Q)^{1-\delta} 
}
\geq
\f12 \f1{1-\delta}
\f{
\lm{Q\cap\{f<\lambda\}}^{\f{1-\delta}n}
}{
\sle(Q)^{1-\delta}  
}
.
\end{align*}
By combining the two previous estimates with \cref{eq:Avolume2}, we conclude that
\begin{align*}
\lm{Q\cap\{f<\lambda\}}^{\f{n-\delta}n}
&\leq
C_2
(1-\delta)
\sum_{k\in\mathbb{N}}
\f
{2^{k(n+\delta)}}
{\sle(Q)^{n+\delta}}
\int_{Q}
\int_{A_k(x)}
|\ind{\{f<\lambda\}} (x)-\ind{\{f<\lambda\}} (y)|
\intd y
\intd x 
,
\end{align*}
where $C_2 = 2^{n+2} C_1 / \sigma_n$.
It follows that
\begin{equation}
\label{eq_smalllambda}
\begin{split}
&\int_{f_Q}^{\max\{m_f,f_Q\}}
\mu(Q \cap\{f>\lambda\})^\frac{1}{q}
\intd\lambda 
\\
&\qquad\leq 
2^\frac{n-\delta}{n}
\frac{
\mu(Q)^\frac{1}{q}
}{
\lm{Q}^\frac{n-\delta}{n}
}
\int_{-\infty}^{m_f}
\lm{Q\cap\{f<\lambda\}}^{\frac{n-\delta}n}
\intd\lambda
\\
&\qquad\leq
2 C_2 (1-\delta)
\frac{
\mu(Q)^\frac{1}{q}
}{
\lm{Q}^\frac{n-\delta}{n}
}
\int_{-\infty}^{m_f}
\sum_{k\in\mathbb{N}}
\f
{2^{k(n+\delta)}}
{\sle(Q)^{n+\delta}}
\int_{Q}
\int_{A_k(x)}
|\ind{\{f<\lambda\}} (x)-\ind{\{f<\lambda\}} (y)|
\intd y
\intd x
\intd \lambda
\\
&\qquad=
2 C_2 (1-\delta)
\frac{
\mu(Q)^\frac{1}{q}
}{
\lm{Q}^\frac{n-\delta}{n}
}
\sum_{k\in\mathbb{N}}
\f
{2^{k(n+\delta)}}
{\sle(Q)^{n+\delta}}
\int_{Q}
\int_{A_k(x)}
\int_{-\infty}^{m_f}
|\ind{\{f<\lambda\}} (x)-\ind{\{f<\lambda\}} (y)|
\intd \lambda
\intd y
\intd x
\\
&\qquad\leq
2 C_2 (1-\delta)
\frac{
\mu(Q)^\frac{1}{q}
}{
\lm{Q}^\frac{n-\delta}{n}
}
\sum_{k\in\mathbb{N}}
\f
{2^{k(n+\delta)}}
{\sle(Q)^{n+\delta}}
\int_{Q}
\int_{A_k(x)}
|f(x)-f(y)|
\intd y
\intd x
\\
&\qquad\leq
2 C_2 (1-\delta)
\frac{
\mu(Q)^\frac{1}{q}
}{
\lm{Q}^\frac{n-\delta}{n}
}
\sum_{k\in\mathbb{N}}
\int_{Q}
\int_{A_k(x)}
\f{
|f(x)-f(y)|
}{
|x-y|^{n+\delta}
}
\intd y
\intd x
\\
&\qquad\leq
2 C_2 (1-\delta)
\frac{
\mu(Q)^\frac{1}{q}
}{
\lm{Q}^\frac{n-\delta}{n}
}
\int_{Q}
\int_{Q\cap B(x,\sle(Q)/2)}
\f{
|f(x)-f(y)|
}{
|x-y|^{n+\delta}
}
\intd y
\intd x
\\
&\qquad\leq
2 C_2 (1-\delta)
\int_{Q}
\int_{Q}
\f{
|f(x)-f(y)|
}{
|x-y|^{n+\delta}
}
\intd y
\, (\Md_{\alpha,Q}\mu)^\frac{1}{q}
\intd x ,
\end{split}
\end{equation}
where 
in the last inequality we used 
\[
\frac{\mu(Q)}{\lm{Q}^{\frac{n-\delta}{n}q}} \leq \Md_{\alpha,Q}\mu(x)
\]
for every $x\in Q$.

It is left to estimate the second term in~\cref{split-fractional}.
In that case, we have \(\lm{Q\cap\{f>\lambda\}}\leq\lm Q/2\) since $\lambda > m_f$.
We apply \cref{C-Z} for $E=\{f>\lambda\}$ on $Q$ at level $\f12$
to obtain a collection \(\{Q_i\}_i\) of Calder\'{o}n--Zygmund cubes with 
\(\sle(Q_i)=2^{-N_i}\sle(Q)\) for some \(N_i\in\mathbb N_0\) 
such that
$Q\cap\{f>\lambda\} \subset \bigcup_i Q_i$ up to a set of Lebesgue measure zero and
\[
\frac{1}{2^{n+1}} < \frac{\lm{Q_i\cap\{f>\lambda\}}}{\lm{Q_i}} \leq \frac{1}{2} .
\]
Fix $i\in\mathbb{N}$ and let $k\geq N_i+1$.
We apply \cref{lem:cla_weakcharf} with $k-N_i$ instead of \(k\) and \(s=1\) for $E = \{f>\lambda\}$ on $Q_i$.
Observe that \(2^{-(k-N_i)}\sle(Q_i)=2^{-k}\sle(Q)\).
For every $k\geq N_i+1$ we obtain
\begin{align*}
\lm{Q_i}^{\f{n-1}n}
&\leq
2^{(n+1)\frac{n-1}{n}}
\lm{Q_i\cap\{f>\lambda\}}^{\f{n-1}n}
\\
&\leq
2^{n}
C_1
\frac{2^{k}}{\sle(Q)}
\int_{Q_i}
\avint_{Q_i\cap A_k(x)}
|\ind{\{f>\lambda\}} (x)-\ind{\{f>\lambda\}} (y)|
\intd y
\intd x
,
\end{align*}
where 
\(
A_k(x)
=
Q \cap  B(x,2^{-k}\sle(Q))\setminus B(x,2^{-k-1}\sle(Q))
\)
as above.
Multiplying both sides by $2^{-k(1-\delta)}$ and summing over $k\geq N_i+1$, we get
\begin{align*}
\sum_{k\geq N_i+1} 
2^{-k(1-\delta)} \lm{Q_i}^{\f{n-1}n} 
\leq 
2^n C_1
\sum_{k\geq N_i+1} 
\frac{2^{k\delta}}{\sle(Q)}
\int_{Q_i}
\avint_{Q_i\cap A_k(x)}
|\ind{\{f>\lambda\}} (x)-\ind{\{f>\lambda\}} (y)|
\intd y
\intd x
.
\end{align*}
We note that
\begin{align*}
\sum_{k\geq N_i+1}
2^{-k(1-\delta)}
&=
\f
{2^{-(N_i+1)(1-\delta)}}
{1-2^{-(1-\delta)}}
=
\f{2^{-(1-\delta)}}{1-2^{-(1-\delta)}}
\f
{\sle(Q_i)^{1-\delta}}
{\sle(Q)^{1-\delta}}
\geq
\f1{2(1-\delta)}
\f
{\sle(Q_i)^{1-\delta}}
{\sle(Q)^{1-\delta}}
.
\end{align*}
By combining the two previous estimates with \cref{eq:Avolume2}, we conclude that
\begin{align*}
\lm{Q_i}^{\f{n-\delta}n}
&\leq C_3 (1-\delta)
\sum_{k\in\mathbb{N}}
\f
{2^{k(n+\delta)}}
{\sle(Q)^{n+\delta}}
\int_{Q_i}
\int_{A_k(x)}
|\ind{\{f>\lambda\}} (x)-\ind{\{f>\lambda\}} (y)|
\intd y
\intd x
,
\end{align*}
where $C_3 = 2^{2n+2} C_1 / \sigma_n$.
Since
\[
\mu(Q \cap\{f>\lambda\})
\leq
\mu\Bigl(\bigcup_iQ_i \Bigr)
\]
by \cref{C-Z}
and
\[
\frac{\mu(Q_i)}{\lm{Q_i}^{\frac{n-\delta}{n}q}} \leq \Md_{\alpha,Q}\mu(x)
\]
for every $x\in Q_i$,
it follows that
\begin{align*}
&\mu(Q \cap\{f>\lambda\})^\frac{1}{q}
\leq
\sum_i \mu(Q_i)^\frac{1}{q} \\
&\qquad\leq
C_3 (1-\delta)
\sum_i
\frac{\mu(Q_i)^\frac{1}{q}}{\lm{Q_i}^\frac{n-\delta}{n}}
\sum_{k\in\mathbb{N}} 
\f
{2^{k(n+\delta)}}
{\sle(Q)^{n+\delta}}
\int_{Q_i}
\int_{A_k(x)}
|\ind{\{f>\lambda\}} (x)-\ind{\{f>\lambda\}} (y)|
\intd y
\intd x \\
&\qquad\leq
 C_3 (1-\delta)
\sum_{k\in\mathbb{N}} 
\f
{2^{k(n+\delta)}}
{\sle(Q)^{n+\delta}}
\sum_i \int_{Q_i}
\int_{A_k(x)}
|\ind{\{f>\lambda\}} (x)-\ind{\{f>\lambda\}} (y)|
\intd y
\, (\Md_{\alpha,Q}\mu(x))^\frac{1}{q}
\intd x \\
&\qquad\leq
 C_3 (1-\delta)
\sum_{k\in\mathbb{N}} 
\f
{2^{k(n+\delta)}}
{\sle(Q)^{n+\delta}}
\int_{Q}
\int_{A_k(x)}
|\ind{\{f>\lambda\}} (x)-\ind{\{f>\lambda\}} (y)|
\intd y
\, (\Md_{\alpha,Q}\mu(x))^\frac{1}{q}
\intd x .
\end{align*}
Integrating both sides in $\lambda$, we obtain
\begin{equation}
\label{eq_largelambda}
\begin{split}
&\int_{\max\{m_f,f_Q\}}^\infty
\mu(Q \cap \{f>\lambda\})^\frac{1}{q}
\intd\lambda
\\
&\qquad\leq
C_3 (1-\delta)
\sum_{k\in\mathbb{N}}
\f
{2^{k(n+\delta)}}
{\sle(Q)^{n+\delta}}
\int_{Q}
\int_{A_k(x)}
\int_{m_f}^\infty
|\ind{\{f>\lambda\}} (x)-\ind{\{f>\lambda\}} (y)|
\intd \lambda
\intd y
\, (\Md_{\alpha,Q}\mu(x))^\frac{1}{q}
\intd x
\\
&\qquad\leq
 C_3 (1-\delta)
\sum_{k\in\mathbb{N}} 
\f
{2^{k(n+\delta)}}
{\sle(Q)^{n+\delta}}
\int_{Q}
\int_{A_k(x)}
|f(x)-f(y)|
\intd y
\, (\Md_{\alpha,Q}\mu(x))^\frac{1}{q}
\intd x
\\
&\qquad\leq
 C_3 (1-\delta)
\sum_{k\in\mathbb{N}}
\int_{Q}
\int_{A_k(x)}
\f{
|f(x)-f(y)|
}{
|x-y|^{n+\delta}
}
\intd y
\, (\Md_{\alpha,Q}\mu(x))^\frac{1}{q}
\intd x
\\
&\qquad\leq
 C_3 (1-\delta)
\int_{Q}
\int_{Q\cap B(x,\sle(Q)/2)}
\f{
|f(x)-f(y)|
}{
|x-y|^{n+\delta}
}
\intd y
\, (\Md_{\alpha,Q}\mu(x))^\frac{1}{q}
\intd x
\\
&\qquad\leq
 C_3 (1-\delta)
\int_{Q}
\int_{Q}
\f{
|f(x)-f(y)|
}{
|x-y|^{n+\delta}
}
\intd y
\, (\Md_{\alpha,Q}\mu(x))^\frac{1}{q}
\intd x .
\end{split}
\end{equation}
By combining the obtained estimates \cref{eq_cavalieriq,split-fractional,eq_smalllambda,eq_largelambda}, we conclude that
\[
\biggl( \int_Q \lvert f-f_Q \rvert^q \intd\mu \biggr)^\frac{1}{q} \leq C (1-\delta)
\int_{Q}
\int_{Q}
\f{
|f(x)-f(y)|
}{
|x-y|^{n+\delta}
}
\intd y
\, (\Md_{\alpha,Q}\mu(x))^\frac{1}{q}
\intd x,
\]
where $C = 4 (2C_2+C_3) = (2^{2n+4}+2^{n+5}) C_1 /\sigma_n $. 
\end{proof}

For a Radon measure satisfying a polynomial growth condition, the following fractional Poincar\'e inequality holds.
\begin{corollary}
\label{cor:fracpoin_growthcond}
Let $0\leq\delta<1$, $1\leq q \leq \frac{n}{n-\delta}$, $\alpha = n-q(n-\delta)$,
$f \in L^1_{\loc}(\mathbb R^n)$
and 
$\mu$ be a Radon measure.
Assume that there exists a constant $C_\mu$ such that
\[
\mu(Q)\leq
C_\mu 
\sle(Q)^{n-\alpha}
\]
for every cube $Q \subset \mathbb R^n$.
Then there exists a dimensional constant $C$ such that
\[
\biggl( \int_Q \lvert f-f_Q \rvert^q \intd\mu \biggr)^\frac{1}{q}
\leq
C_\mu^\frac{1}{q} C
(1-\delta)
\int_Q
\int_Q
\f{|f(x)-f(y)|}{|x-y|^{n+\delta}}
\intd y
\intd x
\]
for every cube $Q \subset \mathbb R^n$.
\end{corollary}

\begin{remark}
We remark that this result combined with Theorem \ref{thm:frac-clas} with $p=1$ yields the classical Meyers--Ziemer theorem \cite{MeyersZiemer1977}. 
\end{remark}

\begin{proof}[Proof of \cref{cor:fracpoin_growthcond}]
Fix a cube $Q\subset\mathbb{R}^n$.
By the assumption, we have
\[
\Md_{\alpha,Q}\mu(x) = 
\sup_{\substack{ Q'\ni x, \\ Q'\in\mathcal D(Q) }}
\sle(Q')^\alpha \f{\mu(Q')}{\lm{Q'}}
\leq
C_\mu
\]
for every $x\in Q$.
Thus, by \cref{thm:wfracpoincare}, the claim follows.
\end{proof}

\section{From fractional $(1,1)$-Poincar\'e inequality to fractional $(q,p)$-Poincar\'e inequality with $A_p$ weights}

\label{sec:fracpoincareAp}

In this section, we show that the fractional $(1,1)$-Poincar\'e inequality implies the fractional $(q,p)$-Poincar\'e inequality. 
Moreover, we are able to obtain the result 
with $A_p$ weights as conjectured in~\cite{hurri2022}.

We recall briefly some concepts about the classes of Muckenhoupt   weights. A weight is a function $w\in L^1_{\mathrm{loc}}(\mathbb{R}^n)$  satisfying $w(x)> 0$ for  almost every point  $x\in\mathbb{R}^n$.

\begin{definition}  Let $w$ be a weight.  
\begin{enumerate}[(i),topsep=5pt,itemsep=5pt]
\item 
We say that $w\in A_1$ if there is a constant $C$ such that
\[
\M w(x) \leq C w(x) 
\]
for almost every $x\in \mathbb{R}^n$.
The $A_1$ constant $[w]_{A_1}$ is defined as the smallest $C$
for which the condition above holds.
\item For $1<p<\infty$ we say that $w\in A_p$ if 
\[
[w]_{A_p} = \sup_{Q} \avint_Q w \intd x \biggl(\avint_Q w^{1-p'} \intd x \biggr)^{p-1}<\infty.
\]
\item The $A_{\infty}$ class is defined as the union of all   the    $A_p$ classes, that is,
\[
A_\infty = \bigcup_{1\leq p<\infty} A_p,
\]
and the $A_\infty$ constant is defined as
\[ 
[w]_{A_\infty} = \sup_{Q} \frac{1}{w(Q)} \int_Q \M(\ind{Q} w) \intd x.
\]
\end{enumerate} 
Recall that for $1\leq r\leq p<\infty$ we have
\begin{equation}
\label{eq:Muckenhouptconstants}
c[w]_{A_\infty}\leq [w]_{A_p} \leq [w]_{A_r} \leq [w]_{A_1} 
\end{equation}
for some \(c>0\) depending only on the dimension.
\end{definition}

We observe that the fractional $(1,1)$-Poincar\'e inequality implies the fractional $(1,p)$-Poincar\'e inequality with $A_p$ weights on the right-hand side.
However, note the extra factor $\delta^{\f1p-1}$ that appears in front.

\begin{corollary}
\label{prop:(1_1)-to-(1_p)}
Let $0<\delta<1$, $1\leq p<\infty$, 
$f \in L^1_{\loc}(\mathbb R^n)$ and $w\in A_p$.
Then there exists a dimensional constant $C$ such that
\[
\avint_Q \lvert f-f_Q \rvert \intd x 
\leq
C
[w]^{\frac{1}{p}}_{A_p}
\f{
(1-\delta)^\f1p
}{
\delta^{1-\f1p}
}
\sle(Q)^\delta
\biggl(
\frac{1}{w(Q)}\int_Q \int_Q
\f{|f(x)-f(y)|^p}{|x-y|^{ n+\delta p}}
\intd y
\,w(x)\intd x
\biggr)^\frac{1}{p}
\]
for every cube $Q\subset\mathbb{R}^n$.
\end{corollary}

\begin{proof}
Let $0\leq\varepsilon\leq\delta$.
By \cref{thm:wfracpoincare} with $\mu=\mathcal{L}$, there exists a constant $C_1$ such that
\begin{equation}
\label{StartingPoint}
\avint_Q \lvert f-f_Q \rvert \intd x 
\leq
C_1
(1-\delta+\varepsilon)
\sle(Q)^{\delta-\varepsilon}
\avint_Q
\int_Q
\f{
|f(x)-f(y)|
}{
|x-y|^{n+\delta-\varepsilon}
}
\intd y
\intd x
.
\end{equation}
If $p=1$, the claim of the \lcnamecref{prop:(1_1)-to-(1_p)} follows from \cref{StartingPoint} with \(\varepsilon=0\) combined with the definition of $A_1$ weights.
It remains to consider $p>1$.
Assume \(0<\varepsilon\leq\delta\) and fix \(x\in Q\).
Then by Hölder's inequality we have
\begin{align*}
\int_Q
\f{|f(x)-f(y)|}{|x-y|^{n+\delta-\varepsilon}}
\intd y
&\leq
\biggl(
\int_Q
\f1{|x-y|^{n-\varepsilon p'}}
\intd y
\biggr)^{\f1{p'}}
\biggl(
\int_Q
\f{|f(x)-f(y)|^p}{|x-y|^{n+\delta p}}
\intd y
\biggr)^{\f1p}
\\
&\leq (n^{\f32}\sigma_n)^{\f1{p'}}
\f{
\sle(Q)^\varepsilon
}{
\varepsilon^{\f1{p'}}
}
\biggl(
\int_Q
\f{|f(x)-f(y)|^p}{|x-y|^{n+\delta p}}
\intd y
\biggr)^{\f1p}
.
\end{align*}
We plug this into \cref{StartingPoint} and apply H\"older's inequality once more with the definition of $A_p$ weights to get
\begin{align*}
\avint_Q \lvert f-f_Q \rvert \intd x 
&\leq C_2
\sle(Q)^\delta
\f{
1-\delta+\varepsilon
}{
\varepsilon^{\f1{p'}}
}
\avint_Q
\biggl(
\int_Q
\f{|f(x)-f(y)|^p}{|x-y|^{n+\delta p}}
\intd y
\biggr)^{\f1p}
\f{w(x)^{\f1p}}{w(x)^{\f1p}}
\intd x
\\
&\leq C_2
\sle(Q)^\delta
\f{
1-\delta+\varepsilon
}{
\varepsilon^{\f1{p'}}
}
\biggl(
\avint_Q
\int_Q
\f{|f(x)-f(y)|^p}{|x-y|^{n+\delta p}}
\intd y
\,
w(x)
\intd x
\biggr)^{\frac1p}
\biggl(
\avint_Q 
w(x)^{1-p'} 
\intd x
\biggr)^{\frac{p-1}{p}}
\\
&\leq C_2
\sle(Q)^\delta
\f{
1-\delta+\varepsilon
}{
\varepsilon^{\f1{p'}}
}
[w]^{\frac{1}{p}}_{A_p}
\biggl(
\frac{1}{w(Q)}
\int_Q
\int_Q
\f{|f(x)-f(y)|^p}{|x-y|^{n+\delta p}}
\intd y
\,
w(x)
\intd x
\biggr)^{\frac1p}
\end{align*}
with \(C_2=C_1\max\{1,n^{\f32}\sigma_n\}\leq92 C_1\).
Setting \(\varepsilon=\min\{\delta,1-\delta\}\) finishes the proof.
\end{proof}

We recall the definitions of the weighted $D_{p}(w)$ and $SD_{p}^{s}(w)$ conditions.

\begin{definition} 
\label{def:smallness}
Let $0<p<\infty$, $0<s<\infty$, $w$ be a weight
and $a:\mathcal{Q}\to [0,\infty)$ be a general functional defined over the collection of all cubes in $\mathbb{R}^n$.

\begin{enumerate}[(i),topsep=5pt,itemsep=5pt]
\item 
The functional $a$ belongs to $D_{p}(w)$ if there is a constant $c$ such that
\begin{equation*}
\Biggl( \sum_i a(Q_i)^p \frac{w(Q_i)}{w(Q)}  \Biggr)^{\frac1p} \leq C a(Q)
\end{equation*}
for any family of disjoint dyadic subcubes $\{ Q_i\}_i $ of any given cube $Q\subset\mathbb{R}^n$.
The smallest constant $C$ above is denoted by $\|a\|_{D_{p}(w)}$. 

\item
The functional $a$ belongs to $SD_{p}^{s}(w)$ if there is a constant $C$ such that
\begin{equation*}
\Biggl( \sum_i a(Q_i)^p \frac{w(Q_i)}{w(Q)}  \Biggr)^{\frac1p} \leq C 
\biggl(\frac{\lm{\bigcup_i Q_i}}{\lm{Q}} \biggr)^{ \frac1s }
a(Q)
\end{equation*}
for any family of disjoint dyadic subcubes $\{ Q_i\}_i $ of any given cube $Q\subset\mathbb{R}^n$.
The smallest constant $C$ above is denoted by $\|a\|_{SD_{p}^s(w)}$.
\end{enumerate}
\end{definition}

The following self-improving property from~\cite[Theorem~1.6]{CantoPerez2021}
is relevant for us.

\begin{theorem}
\label{thm:Automejoraweak}
Let $1<p<\infty$, $w\in A_{\infty}$ and $a\in D_{p}(w)$.
Assume that $f \in L^1_{\loc}(\mathbb R^n)$ such that
\begin{equation*}
\avint_{Q} |f - f_Q| \intd x \leq a(Q) 
\end{equation*}
for every cube $Q\subset\mathbb{R}^n$.
Then there exists a dimensional constant $C$ such that
\begin{equation*}
\lVert f-f_Q \rVert_{L^{p,\infty}( Q, \frac{w \intd x}{w(Q)})} \leq  C p  [w]_{A_{\infty}} \|a\|_{D_p(w)} a(Q).  
\end{equation*}
for every cube $Q\subset\mathbb{R}^n$.
\end{theorem}

For the stronger $SD_{p}^{s}(w)$ condition, we have a better 
self-improvement, see~\cite[Theorem~5.3]{LernerLoristOmbrosi2022}.

\begin{theorem}
\label{thm:PR-strong}

Let $1\leq p<\infty$, $1<s<\infty$, $w$ be a weight and $a\in SD^s_{p}(w)$.
Assume that $f \in L^1_{\loc}(\mathbb R^n)$ such that
\begin{equation*}
\avint_{Q} |f-f_{Q}| \intd x \le a(Q)
\end{equation*}
for every cube $Q\subset\mathbb{R}^n$. Then there exists a dimensional constant $C$ such that
\begin{equation*}
\lVert f-f_Q \rVert_{ L^{p}(Q,\frac{w \intd x}{w(Q)}) } \leq  C s \|a\|_{SD_{p}^s(w)} a(Q)
\end{equation*} 
for every cube $Q\subset\mathbb{R}^n$.

\end{theorem}

Another important tool that we need is the following fractional truncation method
which can be shown by adapting the proof of~\cite[Theorem~4.1]{DydaLizavetaVahakangas2016}.

\begin{theorem}
\label{t.truncation}
Let $0<\delta<1$, $1\leq p\leq q<\infty$,
$f \in L^1_{\loc}(\mathbb R^n)$ and 
$w$ be a weight.
Then the following conditions are equivalent.
\begin{enumerate}[(i)]
\item 
There is a constant $C_1$ such that
\begin{align*}
\inf_{c\in\mathbb{R}}
\lVert f-c \rVert_{L^{q,\infty}( Q, \frac{w \intd x}{w(Q)})} 
\leq C_1
\biggl(
\frac{1}{w(Q)}
\int_{Q}\int_{Q} \frac{\vert f(x)-f(y)\vert^p}{\lvert x-y\rvert^{n+\delta p}}
\intd y \, 
w(x)
\intd x
\biggr)^\frac{1}{p}
\end{align*}
for every cube $Q\subset\mathbb{R}^n$.
\item
There is a constant $C_2$ such that 
\begin{align*}
\inf_{c\in\mathbb{R}} 
\|f-c\|_{ L^{q}(Q,\frac{w \intd x}{w(Q)}) }
\leq C_2
\biggl(
\frac{1}{w(Q)}
\int_{Q}\int_{Q} \frac{\vert f(x)-f(y)\vert^p}{\lvert x-y\rvert^{n+\delta p}} 
\intd y \, 
w(x)
\intd x \biggr)^\frac{1}{p}
\end{align*}
for every cube $Q\subset\mathbb{R}^n$.
\end{enumerate}
Moreover, in 
the implication from (i) to (ii) the constant $C_2$ is of the form $CC_1$, where $C$ only depends on the dimension, 
and in the implication from (ii) to (i) we have
$C_1=C_2$.
\end{theorem}

We are ready to state and prove the main results of this section, which are the fractional $(q,p)$-Poincar\'e inequalities with $A_p$ weights.
These results extend Theorems~2.1~and~2.3 in \cite{hurri2022}.
We emphasize that the factor $(1-\delta)^{\frac1p}$ remains despite the singularity introduced by the weight.

\begin{theorem}\label{selfIMproveBadConstant} 
Let $0<\delta<1$, $1\leq r \leq p < \frac{n}{\delta}$, 
$f \in L^1_{\loc}(\mathbb R^n)$ and
$w \in A_r$.
Let
$q$ be 
defined by
\begin{equation*} 
\frac{1}{p} -\frac{1}{q }=\frac{\delta}{nr}.
\end{equation*}
Then there exists a dimensional constant $C$ such that  
\begin{align*}
&
\inf_{c\in \mathbb{R}}
\biggl(\frac{1}{w(Q)} \int_Q 
|f-c|^{q}
\, w 
\intd x
\biggr)^{\frac{1}{q}} 
\\&\qquad \leq 
C
q
[w]_{A_p}^{ \frac1p } [w]_{A_r}^{ \frac{\delta}{nr} } [w]_{A_{\infty}}
\f{(1-\delta)^{\frac1p} 
}{
\delta^{1-\f1p} 
}
\sle(Q)^{\delta} 
\biggl(\frac{1}{w(Q)} \int_{Q}  \int_{Q}  \frac{|f(x)- f(y)|^p}{|x-y|^{n+\delta p}} \intd y \, 
w(x)
\intd x \biggr)^{\frac1p}
\end{align*}
for every cube $Q\subset\mathbb{R}^n$.
\end{theorem}

\begin{remark}
We remark that we could replace $[w]_{A_p}^{ \frac1p } [w]_{A_r}^{ \frac{\delta}{nr} } [w]_{A_{\infty}}$ by a multiple of 
$[w]_{A_r}^{ \frac1p+\frac{\delta}{nr}+1 }$.
\end{remark}

\begin{proof}[Proof of \cref{selfIMproveBadConstant}]

Denote
\[
a_f(Q)
=
C_1
[w]^{\frac{1}{p}}_{A_p}
\f{
(1-\delta)^\f1p
}{
\delta^{1-\f1p} 
}
\sle(Q)^{\delta} \biggl( \frac{1}{w(Q)} \int_{Q}  \int_{Q} 
\frac{|f(x)- f(y)|^p}{|x-y|^{n+\delta p}}
\intd y \, 
w(x)
\intd x 
\biggr)^{\frac1p}
,
\]
where $C_1$ is the dimensional constant in \cref{prop:(1_1)-to-(1_p)}.
By \cref{prop:(1_1)-to-(1_p)}, it holds that 
\[
\avint_{Q} |f - f_Q| \intd x \leq a_f(Q) .
\]
In addition,
by~\cite[Lemma~3.3]{CejasMosqueraPerezRela2021} (which also holds for $M=1$ corresponding to our case), we have $a_f\in D_{q}(w)$ 
such that
\[
\|a_f\|_{D_{q}(w)} \leq [w]_{A_r}^{ \frac{\delta}{nr} }, 
\]
uniformly in $f$.
Hence, we may apply \cref{thm:Automejoraweak} to obtain
\begin{align*}
& \| f-f_Q\|_{L^{q,\infty} ( Q, \frac{w \intd x}{w(Q)})} 
\leq 
C_2
q
[w]_{A_{\infty}} [w]_{A_r}^{ \frac{\delta}{nr}} a_f(Q)
\\
&\qquad = 
C 
q
[w]_{A_{\infty}} [w]^{\frac{\delta}{nr}}_{A_r} [w]_{A_p}^{ \frac1p }
\f{
(1-\delta)^{\frac1p}  
}{
\delta^{1-\f1p1} 
}
\sle(Q)^{\delta}
\biggl( 
\frac{1}{w(Q)} \int_{Q}  \int_{Q}  \frac{|f(x)- f(y)|^p}{|x-y|^{n+\delta p}} \intd y \,
w(x)
\intd x \biggr)^{\frac1p}
,
\end{align*}
where $C_2$ is the constant in \cref{thm:Automejoraweak}
and $C=C_1 C_2$.
An application of \cref{t.truncation} finishes the proof.
\end{proof}

A better dependency on the $A_p$ constants in front can be attained at the expense of having a smaller borderline exponent.

\begin{theorem} 
\label{selfIMproveGoodConstant}  
Let $0<\delta<1$, $1\leq r \leq p< \frac{n}{\delta}$,
$f \in L^1_{\loc}(\mathbb R^n)$
and $w\in A_r$.
Let $q$ be 
defined by
\begin{equation*}
\frac{1}{p} -\frac{1}{q }=   \frac{\delta}{n} \frac{1}{r+\log [w]_{A_r}}.
\end{equation*}
Then there exists a dimensional constant $C$ such that
\begin{align*}
\inf_{c\in \mathbb{R}}
& \biggl( \frac{1}{w(Q)} 
\int_{ Q } 
|f-c|^{q} 
\,w \intd x
\biggr)^{\frac{1}{q}}  
\\
&\qquad\leq 
C 
\frac{npr}{nr-\delta p}
[w]_{A_p}^{ \frac1p }  [w]_{A_{\infty}}
\f{
(1-\delta)^{\frac1p}
}{
\delta^{1-\f1p} 
}
\sle(Q)^{\delta}
\biggl(   \frac{1}{w(Q)} \int_{Q}  \int_{Q}  \frac{|f(x)- f(y)|^p}{|x-y|^{n+\delta p}} \intd y \, 
w(x)
\intd x\biggr)^{\frac1p}
\end{align*}
for every cube $Q\subset\mathbb{R}^n$.
\end{theorem}

\begin{proof}

Denote 
\[
a_f(Q)
=
C_1
[w]^{\frac{1}{p}}_{A_p}
\f{
(1-\delta)^\f1p
}{
\delta^{1-\f1p} 
}
\sle(Q)^{\delta} \biggl( \frac{1}{w(Q)} \int_{Q}  \int_{Q} 
\frac{|f(x)- f(y)|^p}{|x-y|^{n+\delta p}}
\intd y \, 
w(x)
\intd x 
\biggr)^{\frac1p}
,
\]
where $C_1$ is the constant in \cref{prop:(1_1)-to-(1_p)}.
By \cref{prop:(1_1)-to-(1_p)}, it holds that
\[
\avint_{Q} |f - f_Q| \intd x \leq a_f(Q) .
\]
We distinguish between the cases $[w]_{A_r} > e^{\frac{1}{\delta}}$ and the opposite.
Assume first that $[w]_{A_r} > e^{\frac{1}{\delta}}$.
By~\cite[Lemma~6.2]{CejasMosqueraPerezRela2021}, we have $a\in SD_{q}^{s}(w)$ 
with $M=1+ \frac{1}{r}\log[w]_{A_{r}}$ and $s=\frac{nM'}{\delta} > 1$,
such that
\begin{equation*}
\|a_f\|_{SD_{q}^{s}(w)} \leq [w]_{A_r}^{ \frac{ \delta}{nrM} },
\end{equation*}
uniformly in $f$.
Hence, applying \cref{thm:PR-strong} with $q$,
we obtain  
\begin{align*}
\lVert f-f_Q \rVert_{ L^{q}(Q,\frac{w \intd x}{w(Q)}) }
&\leq 
C_2 \frac{nM'}{\delta} 
[w]_{A_r}^{\frac{\delta}{nrM}} a_f(Q) 
\leq  C_2  \frac{nM'}{\delta}  [w]_{A_r}^{\frac{1}{ r+\log [w]_{A_r}  }} a_f(Q) 
\\
&\leq  C_2 \frac{n}{\delta} \frac{r+ \log[w]_{A_r}}{ \log[w]_{A_r}} e^1 a_f(Q)
,
\end{align*}
where $C_2$ is the constant in \cref{thm:PR-strong}.
By the assumption $[w]_{A_r} > e^{\frac{1}{\delta}}$, we have 
\begin{align*}
\lVert f-f_Q \rVert_{ L^{q}(Q,\frac{w \intd x}{w(Q)}) }
&\leq
C (r+ \log[w]_{A_r}) a_f(Q) 
\leq C r [w]_{A_r}  a_f(Q)
,
\end{align*}
where $C = C_1 C_2 n e^1 $.
This gives the claim when $[w]_{A_r} > e^{\frac{1}{\delta}}$. 

Assume now that $[w]_{A_r} \leq  e^{\frac{1}{\delta}}$.  
By~\cite[Lemma~6.2]{CejasMosqueraPerezRela2021}, we have $a_f\in D_{m}(w)$ 
such that
\[
\|a_f\|_{D_{m}(w)} \leq [w]_{A_r}^{ \frac{\delta}{nr}}
\leq
e^\frac{1}{nr}
,
\]
where the exponent $m$ is defined by
$\frac{1}{p} -\frac{1}{ m}=\frac{\delta}{nr}$.
Applying \cref{thm:Automejoraweak}, we get
\begin{align*}
\lVert f-f_Q \rVert_{L^{m,\infty}( Q, \frac{w \intd x}{w(Q)})} 
&\leq 
C_3 
m
[w]_{A_{\infty}} e^{\frac{1}{nr}} a_f(Q)
\leq C 
m
[w]_{A_r} a_f(Q) ,
\end{align*}
where $C_3$ is the constant in \cref{thm:Automejoraweak}, $C_4$ is the constant in \cref{eq:Muckenhouptconstants} and $C=C_3 C_4 e^1$.
Since $q \leq m$,
Jensen's inequality implies
\begin{align*}
\lVert f-f_Q \rVert_{L^{q,\infty}( Q, \frac{w \intd x}{w(Q)})} & \leq
\lVert f-f_Q \rVert_{L^{m,\infty}( Q, \frac{w\intd x}{w(Q)})}
\leq
C m [w]_{A_{r}} a_f(Q).  
\end{align*}
An application of \cref{t.truncation} finishes the proof.
\end{proof}

\section{From weighted fractional to weighted classical Poincar\'e inequality}
\label{sec:weight_to_clas}

This section shows that \cref{thm:wfracpoincare} implies the corresponding weighted classical Poincar\'e inequality \cref{cor:weight-clas-poincare}.
For any \(0<\alpha\leq n\) the Riesz potential $\I_\alpha$ of a Radon measure \(\mu\) is
\[
\I_\alpha \mu(x) 
=  
\int_{\mathbb{R}^n}
\frac{\intd\mu(y)}{|x-y|^{n-\alpha}}
\]
for every $x\in\mathbb{R}^n$.
The following \lcnamecref{lemma:riesz-bound} is an improved version of the well-known result that the Riesz potential is bounded by the maximal function.

\begin{lemma}
\label{lemma:riesz-bound}
Let $Q\subset\mathbb{R}^n$ be a cube, $\mu$ be a Radon measure and $0<\alpha<n$.
Then
\[
\I_\alpha (\ind{Q} \mu)(x)
\leq \frac{2^{n-\alpha}n}{\alpha} 
\mu(Q)^{\frac{\alpha}{n}}(\M\mu(x))^{1-\frac{\alpha}{n}}
\]
for every $x\in Q$.
\end{lemma}

\begin{proof} 
Let $Q\subset\mathbb{R}^n$ be a fixed cube and $x\in Q$.
For $t>0$ let
\(
Q_{x,t}
\)
be the cube with center at $x$ and side lenght $2t^{-\frac{1}{n-\alpha}}$. 
Then using Cavalieri's principle, we obtain
\begin{align*}
\int_{Q}\frac{\intd \mu(y)}{\lvert x-y \rvert^{n-\alpha}}
&= 
\int_0^{\infty} \mu\biggl(\Bigl\{y\in Q: \frac{1}{\lvert x-y \rvert^{n-\alpha}}>t\Bigr\}\biggr) \intd t
\\
&= \int_0^{\infty} \mu\bigl(\bigl\{ y\in Q: \lvert x-y \rvert<t^{-\frac{1}{n-\alpha}} \bigr\}\bigr) \intd t \\
&\leq \int_0^{\infty} \min\biggl\{\mu(Q),
\frac{\mu(Q_{x,t})}{\lm{Q_{x,t}}} \lm{Q_{x,t}}\biggr\} \intd t .\\
&\leq \int_0^{\infty}\min\bigl\{\mu(Q),\M\mu(x) 2^nt^{-\frac{n}{n-\alpha}}\bigr\} \intd t
\\
&=\int_0^{2^{n-\alpha}(\M\mu(x)/\mu(Q))^\frac{n-\alpha}{n}}\mu(Q)
\intd t
+ 2^n \int_{2^{n-\alpha}(\M\mu(x)/\mu(Q))^\frac{n-\alpha}{n}}^\infty \M\mu(x)t^{-\frac{n}{n-\alpha}}
\intd t
\\
&=
\frac{2^{n-\alpha}n}{\alpha}
\mu(Q)^\frac{\alpha}{n}(\M\mu(x))^\frac{n-\alpha}{n}.
\end{align*}
Thus, the claim holds.
\end{proof}

The next \lcnamecref{thm:frac-clas} states that the weighted fractional term can be bounded by the weighted gradient term, but we have the maximal function of the measure on the right hand side.
We are mainly interested in the case \(p=1\). This theorem improves Theorems 2.1 and 2.2 from \cite{hurri2023}.

\begin{theorem}
\label{thm:frac-clas}
Let \(1\leq p<\infty\), \(\f{p-1}p< \delta<1\), $f\in W^{1,p}_\loc(\mathbb{R}^n)$ and $\mu\ll\mathcal L$ be a Radon measure.
Then 
\begin{align*}
\int_Q
\int_Q
\f{|f(x)-f(y)|^p}{|x-y|^{n+\delta p}}
\intd y
\intd \mu(x)
&\leq
\frac{2^{n-(1-\delta)p}n}{(1-\delta)p}
\frac{\mu(Q)^{\frac{(1-\delta)p}{n}}}{1-(1-\delta)p}
\int_Q
|\nabla f|^p
\,
(\M \mu)^{1-\frac{(1-\delta)p}{n}}
\intd x
\\
\intertext{
for every cube $Q\subset\mathbb{R}^n$.
As a direct consequence,
}
\int_Q
\int_Q
\f{|f(x)-f(y)|^p}{|x-y|^{n+\delta p}}
\intd y
\intd \mu(x)
&\leq
\frac{2^{n-(1-\delta)p}n}{(1-\delta)p}
\frac{\sle(Q)^{(1-\delta)p}}{1-(1-\delta)p}
\int_Q
|\nabla f|^p
\,
\M \mu
\intd x 
.
\end{align*}

Alternatively, we can assume that $\mu$ is a general Radon measure and the claim holds for any continuous function $f\in W^{1,p}_\loc(\mathbb{R}^n)$.
\end{theorem}

\begin{proof}
The second inequality follows from the first inequality due to the fact that \(\mu(Q)/\sle(Q)^n\leq\M\mu(x)\) for any \(x\in Q\).
It remains to prove the first inequality.
If \(f\) is continuously differentiable then by the Fundamental Theorem of Calculus we have
\[
f(y)-f(x) = \int_0^1 \nabla f(x+t(y-x)) \cdot (y-x) \intd t
\]
for every \((x,y)\in Q\times Q\).
If \(f\) is a Sobolev function then the previous equality still holds
for almost every \((x,y)\in Q\times Q\).
Then by H\"older's inequality, it holds that
\[
|f(x)-f(y)|^p \leq \int_0^1 |\nabla f(x+t(y-x))|^p |x-y|^p \intd t .
\]
Applying this with 
Fubini's theorem
and doing the change of variables $y\mapsto z = x+t(y-x)$,
we get
\begin{align*}
&\int_Q
\int_Q
\f{|f(x)-f(y)|^p}{|x-y|^{n+\delta p}}
\intd y
\intd \mu(x)
\\
&\qquad\leq
\int_Q
\int_0^1 
\int_{Q}
\f{|\nabla f(x+t(y-x))|^p  }{|x-y|^{n-(1-\delta)p}}
\intd y
\intd t
\intd \mu(x)
\\
&\qquad =
\int_Q
\int_0^1 
\int_{(1-t)x+tQ}
\f{|\nabla f(z)|^p  }{|x-z|^{n-(1-\delta)p}}
\frac{t^{n-(1-\delta)p}}{t^n}
\intd z
\intd t
\intd \mu(x)
\\
&\qquad\leq
\int_Q
\int_Q
\f{|\nabla f(z)|^p  }{|x-z|^{n-(1-\delta)p}}
\int_{0}^1 
\f1{
t^{(1-\delta)p}
}
\intd t
\intd z
\intd \mu(x)
\\
&\qquad =
\frac{1}{1-(1-\delta)p}
\int_Q
|\nabla f(z)|^p
\int_Q
\f{ \mu(x) }{|x-z|^{n-(1-\delta)p}}
\intd x
\intd z .
\end{align*}
Here we used $(1-t)x+tQ \subset Q$ for $x\in Q$
and
\((1-\delta)p<1\).
By applying \cref{lemma:riesz-bound}, we obtain
\begin{align*}
\int_Q
\f1{|x-z|^{n-(1-\delta)p}}
\intd \mu(x)
&=
\I_{(1-\delta)p} 
( \ind{Q} \mu )(z) 
\leq 
\frac{2^{n-(1-\delta)p}n}{(1-\delta)p}
\mu(Q)^{\frac{(1-\delta)p}{n}} 
(\M \mu(z))^{1-\frac{(1-\delta)p}{n}}
\end{align*}
for every $z\in Q$.
Hence, we conclude that
\begin{align*}
\int_Q
\int_Q
\f{|f(x)-f(y)|^p}{|x-y|^{n+\delta p}}
\intd y
\intd \mu(x)
&\leq
\frac{2^{n-(1-\delta)p}n}{(1-\delta)p}
\frac{\mu(Q)^{\frac{(1-\delta)p}{n}}}{1-(1-\delta)p}
\int_Q
|\nabla f(z)|^p
\,
(\M \mu(z))^{1-\frac{(1-\delta)p}{n}}
\intd z .
\end{align*}
This completes the proof.
\end{proof}

For $A_1$ weights,
we can replace the maximal function in \cref{thm:frac-clas} by the weight itself.

\begin{corollary}
\label{cor:frac-clas-A1}
Let $\f{p-1}p<\delta<1$, $f\in W^{1,1}_\loc(\mathbb{R}^n)$ and $w\in A_1$. Then 
\begin{align*}
\int_Q
\int_Q
\f{|f(x)-f(y)|^p}{|x-y|^{n+\delta p}}
\intd y
\, 
w(x)
\intd x
&\leq
\frac{2^{n-(1-\delta)p}n}{(1-\delta)p}
\f{w(Q)^{\frac{(1-\delta)p}{n}}}{1-(1-\delta)p}
[w]_{A_1}^{1-\frac{(1-\delta)p}{n}}
\int_Q
|\nabla f|^p
\,
w^{1-\frac{(1-\delta)p}{n}}
\intd x 
\\
\intertext{
for every cube $Q\subset\mathbb{R}^n$.
As a direct consequence,
}
\int_Q
\int_Q
\f{|f(x)-f(y)|^p}{|x-y|^{n+\delta p}}
\intd y
\, 
w(x)
\intd x
&\leq
\frac{2^{n-(1-\delta)p}n}{(1-\delta)p}
\f{\sle(Q)^{(1-\delta)p}}{1-(1-\delta)p}
[w]_{A_1}
\int_Q
|\nabla f|^p
\,
w
\intd x
.
\end{align*}
\end{corollary}

\begin{proof}
The second inequality follows from the first inequality due to the fact that \(w(Q)/\lm Q\leq [w]_{A_1}w(x)\) for any \(x\in Q\).
It remains to prove the first inequality.
By \cref{thm:frac-clas} and the definition of $A_1$ weights,
we get
\begin{align*}
\int_Q
\int_Q
\f{|f(x)-f(y)|^p}{|x-y|^{n+\delta p}}
\intd y
\,
w(x)
\intd x
&\leq
\frac{2^{n-(1-\delta)p}n}{(1-\delta)p}
\frac{w(Q)^{\frac{(1-\delta)p}{n}}}{1-(1-\delta)p}
\int_Q
|\nabla f|^p
\,
(\M w)^{1-\frac{(1-\delta)p}{n}}
\intd x
\\
&\leq
\frac{2^{n-(1-\delta)p}n}{(1-\delta)p}
\frac{w(Q)^{\frac{(1-\delta)p}{n}}}{1-(1-\delta)p}
[w]_{A_1}^{1-\frac{(1-\delta)p}{n}}
\int_Q
|\nabla f|^p
\,
w^{1-\frac{(1-\delta)p}{n}}
\intd x .
\end{align*}
\end{proof}

The next \lcnamecref{coarea} is the coarea formula for Sobolev functions~\cite[Proposition~3.2]{swanson2007}.
\begin{lemma}
\label{coarea}
Let $f\in W_\loc^{1,1}(\mathbb R^n)$ and let $g: \mathbb R^n \to \mathbb R_+$ be a measurable function.
Then
\[
\int_{E}
|\nabla f(x)| \, g(x)
\intd x
=
\int_{-\infty}^\infty
\int_{E \cap \mb \{f>\lambda\}}
g(x)
\intd\sm x
\intd\lambda
\]
for every Lebesgue measurable set $E \subset \mathbb R^n$.
\end{lemma}

Combining \cref{thm:wfracpoincare} with \cref{cor:frac-clas-A1} we obtain the corresponding weighted classical Poincar\'e inequality. 
For thoroughness, we also give another, direct proof for \cref{cor:weight-clas-poincare}
by applying the coarea formula and the relative isoperimetric inequality (\cref{rel.iso.ineq}) instead of \cref{lem:cla_weakcharf}.

\begin{corollary}
\label{cor:weight-clas-poincare}
Let $1\leq q \leq \frac{n}{n-1}$, $\alpha = n-q(n-1)$,
$f\in W^{1,1}_\loc(\mathbb{R}^n)$ and let
$\mu$ be a Radon measure with $\mu\ll\mathcal{L}$.
There exists a dimensional constant $C$ such that
\[
\biggl( \int_Q \lvert f-f_Q \rvert^q \intd\mu \biggr)^\frac{1}{q}
\leq
C
\int_Q
|\nabla f|
\,
(\Md_{\alpha,Q}\mu)^\frac{1}{q}
\intd x
\]
for every cube $Q\subset\mathbb{R}^n$.

Alternatively, we can assume that $\mu$ is a general Radon measure and the claim holds for any continuous function \(f\in W^{1,1}(\mathbb R^n)\).
\end{corollary}

\begin{proof}[Proof 1]
We prove the claim first
for \(1<q\leq\f n{n-1}\), which means \(0\leq\alpha<1\).
For any \(\alpha<\delta<1\) let \(q_\delta=\f{n-\alpha}{n-\delta}\).
Note that \(q_\delta\to q\) for \(\delta\to 1\).
The function $(\Md_{\alpha,Q}\mu)^\frac{1}{q_\delta}$ is an $A_1$ weight with 
\[
\bigl[(\Md_{\alpha,Q}\mu)^\frac{1}{q_\delta}\bigr]_{A_1}
=
\bigl[(\Md_{\alpha,Q}\mu)^{\f n{n-\alpha}-\f\delta{n-\alpha}}\bigr]_{A_1}
\leq
\f{
15^n4n
}
\delta
\]
by for example~\cite[Lemma~3.6]{hurri2022},
and we may apply 
\cref{cor:frac-clas-A1} with \(p=1\) to get
\begin{align*}
\int_Q
\int_Q
\f{|f(x)-f(y)|}{|x-y|^{n+\delta}}
\intd y
\, (\Md_{\alpha,Q}\mu(x))^\frac{1}{q_\delta}
\intd x
&\leq
\frac{30^nn^22^{1+\delta}}{\delta^2(1-\delta)}
\sle(Q)^{1-\delta}
\int_Q
|\nabla f|
\,
(\Md_{\alpha,Q}\mu)^\frac{1}{q_\delta}
\intd x .
\end{align*}
Then \cref{thm:wfracpoincare} further implies 
that there exists a constant $C_1$ such that
\begin{align*}
\biggl( \int_Q \lvert f-f_Q \rvert^{q_\delta} \intd\mu \biggr)^\frac{1}{q_\delta}
&\leq
C_1
(1-\delta)
\int_Q
\int_Q
\f{|f(x)-f(y)|}{|x-y|^{n+\delta}}
\intd y
\, (\Md_{\alpha,Q}\mu(x))^\frac{1}{q_\delta}
\intd x
\\
&\leq
\frac{C}{\delta^2}
\sle(Q)^{1-\delta}
\int_Q
|\nabla f|
\,
(\Md_{\alpha,Q}\mu)^\frac{1}{q_\delta}
\intd x ,
\end{align*}
where $C = 30^n 4 n^2 C_1 $.
For any function \(g:\mathbb R^n\to\mathbb R\) the restriction \(g^{q_\delta}\ind{g\leq1}\) is bounded by \(1\), and \(g^{q_\delta}\ind{g>1}\) is pointwise increasing in \(\delta\).
Thus, by the monotone and the dominated convergence theorem both sides of the previous display converge for \(\delta\to1\) to the desired limit, concluding the proof for \(q>1\).

In order to prove the claim for \(q=1\) we have to differentiate between the two alternatives in the assumptions of the \lcnamecref{cor:weight-clas-poincare}.
We first consider the case that \(f\) is a Sobolev function and \(\mu\) is absolutely continuous.
Then \(\mu\) has a density function $v\in L_\loc^1(\mathbb{R}^n)$ by the Radon--Nikodym theorem.
Let \(k\in\mathbb N\) and denote by \(\mu_k\) the truncated measure that has the bounded density $\min\{v,k\}$.
Then \((\Md_{\alpha,Q}\mu_k(x))^{\f1q}\) is uniformly bounded in \(\alpha\), \(q\) and \(x\) and converges pointwise to \(\Md_{1,Q}\mu_k(x)\) for \(q\to 1\).
Thus by Fatou's lemma and the dominated convergence theorem we have
\begin{equation}
\label{eq_qto1}
\begin{split}
\int_Q \lvert f-f_Q \rvert \intd\mu_k
&\leq
\liminf_{q\to 1}
\biggl( \int_Q \lvert f-f_Q \rvert^q \intd\mu_k \biggr)^\frac{1}{q}
\\
&\leq
C
\liminf_{q\to 1}
\int_Q
|\nabla f|
\,
(\Md_{\alpha,Q}\mu_k)^\frac{1}{q}
\intd x
\\
&=
C
\int_Q
|\nabla f|
\,
\Md_{1,Q}\mu_k
\intd x
.
\end{split}
\end{equation}
Because \(\mu_k\) converges to \(\mu\) and \(\Md_{1,Q}\mu_k\) converges to \(\Md_{1,Q}\mu\) pointwise monotonously from below
we can use the monotone convergence theorem
to conclude from the previous display that
\begin{equation}
\label{eq_ktoinfty}
\begin{split}
\int_Q \lvert f-f_Q \rvert \intd\mu
&=
\lim_{k\to\infty}
\int_Q \lvert f-f_Q \rvert \intd\mu_k
\\
&\leq
C
\lim_{k\to\infty}
\int_Q
|\nabla f|
\,
\Md_{1,Q}\mu_k
\intd x
\\
&=
C
\int_Q
|\nabla f|
\,
\Md_{1,Q}\mu
\intd x
,
\end{split}
\end{equation}
finishing the proof for \(q=1\) in the case that \(f\in W^{1,1}_\loc(\mathbb R^n)\) and \(\mu\ll\mathcal L\).

In the case that \(f\) is continuous and \(\mu\) is a general Radon measure  the proof goes the same, except we let \(\mu_k\) not be a truncation, but instead the measure that averages \(\mu\) over dyadic cubes of scale \(2^{-k}\sle(Q)\), i.e.
\[
\mu_k(E)
=
\sum_{P\in\mathcal D_k(Q)}
\lm{E\cap P}
\f{\mu(P)}{\lm P}
.
\]
Then \((\Md_{\alpha,Q}\mu_k(x))^{\f1q}\) is uniformly bounded in \(\alpha\), \(q\) and \(x\) and converges pointwise to \(\Md_{1,Q}\mu_k(x)\) for \(q\to 1\), which means we can conclude \cref{eq_qto1} also in this case.
Also \(\Md_{1,Q}\mu_k\) converges to \(\Md_{1,Q}\mu\) pointwise monotonously from below.
Furthermore, \(\mu_k\) converges to \(\mu\) weakly, see \cite[Theorem~1.40]{evansgariepy}.
Thus, we can conclude \cref{eq_ktoinfty} also in this case, finishing the proof for \(q=1\) also in the case that \(f\in W^{1,1}_\loc(\mathbb R^n)\) is continuous and \(\mu\) is a general Radon measure.
\end{proof}

\begin{proof}[Proof 2]
Fix $Q \subset \mathbb R^n$ and denote $\Omega_\lambda = \{x\in Q: \lvert f-f_Q \rvert > \lambda \}$.
As in the proof of \cref{thm:wfracpoincare}, we reduce the problem to bounding the sum
\begin{equation}
\label{split}
\int_{f_Q}^{\max\{m_f,f_Q\}}
\mu(Q \cap \{f > \lambda \})^\frac{1}{q}
\intd\lambda 
+
\int_{\max\{m_f,f_Q\}}^\infty
\mu(Q \cap \{f > \lambda \})^\frac{1}{q} \intd\lambda
,
\end{equation}
and estimate the first summand by
\begin{align*}
2^\frac{n-1}{n}
\frac{
\mu(Q)^\frac{1}{q}
}{
\lm{Q}^\frac{n-1}{n}
}
\int_{-\infty}^{f_Q}
\lm{Q\cap\{f<\lambda\}}^{\frac{n-1}n}
\intd\lambda
&\leq 2 C_1
\frac{
\mu(Q)^\frac{1}{q}
}{
\lm{Q}^\frac{n-1}{n}
}
\int_{-\infty}^{f_Q}
\sm{Q\cap\mb{\{f<\lambda\}}}
\intd\lambda
\\
&\leq 2 C_1
\frac{
\mu(Q)^\frac{1}{q}
}{
\lm{Q}^\frac{n-1}{n}
}
\int_Q|\nabla f|
\intd x
\\
&\leq 2 C_1
\int_Q
|\nabla f| \, (\Md_{\alpha,Q}\mu)^\frac{1}{q}
\intd x
,
\end{align*}
where we used \(\lm{Q\cap\{f<\lambda\}}\leq\lm Q/2\), \cref{rel.iso.ineq}, \cref{coarea} and
\[
\frac{\mu(Q)^\frac{1}{q}}{\lm{Q}^\frac{n-1}{n}} \leq (\Md_{\alpha,Q}\mu(x))^\frac{1}{q}
\]
for every $x\in Q$.

It is left to estimate the second term in~\cref{split}.
In that case, we have \(\lm{Q\cap\{f>\lambda\}}\leq\lm Q/2\) since $\lambda > m_f$.
We apply \cref{C-Z} for $E=\{f>\lambda\}$ on $Q$ at level $\f12$
to obtain a collection \(\{Q_i\}_i\) of Calder\'{o}n--Zygmund cubes
such that
$Q\cap\{f>\lambda\} \subset \bigcup_i Q_i$ up to a set of Lebesgue measure zero and
\[
\frac{1}{2^{n+1}} < \frac{\lm{Q_i\cap\{f>\lambda\}}}{\lm{Q_i}} \leq \frac{1}{2} .
\]
By \cref{rel.iso.ineq}, we have
\[
\lm{Q_i}^\frac{n-1}{n}
\leq 
2^{(n+1)\frac{n-1}{n}} \lm{Q_i\cap \{f>\lambda\}}^\frac{n-1}{n}
\leq
C_2 \sm{Q_i \cap\mb\{f>\lambda\}},
\]
where $C_2 = 2^n C_1$.
Since
\[
\mu(Q \cap\{f>\lambda\})
\leq
\mu\Bigl(\bigcup_iQ_i \Bigr)
\]
by \cref{C-Z}
and
\[
\frac{\mu(Q_i)}{\lm{Q_i}^{\frac{n-1}{n}q}} \leq \Md_{\alpha,Q}\mu(x)
\]
for every $x\in Q_i$,
it follows that
\begin{align*}
\mu(Q \cap\{f>\lambda\})^\frac{1}{q} &\leq \sum_i \mu(Q_i)^\frac{1}{q} \\
&\leq C_2 \sum_i \frac{\mu(Q_i)^\frac{1}{q}}{\lm{Q_i}^\frac{n-1}{n}} \sm{Q_i \cap\mb\{f>\lambda\}} \\
&\leq C_2 \sum_i \int_{Q_i \cap\mb\{f>\lambda\}} (\Md_{\alpha,Q}\mu)^\frac{1}{q} \intd\mathcal{H}^{n-1} \\
&\leq C_2 \int_{Q \cap\mb\{f>\lambda\}} (\Md_{\alpha,Q}\mu)^\frac{1}{q} \intd\mathcal{H}^{n-1} .
\end{align*}
Integrating both sides in $\lambda$ and applying \cref{coarea}, we obtain
\begin{align*}
\int_{\max\{m_f,f_Q\}}^\infty
\mu(Q \cap \{f>\lambda\})^\frac{1}{q}
\intd\lambda
&\leq
C_2
\int_{\max\{m_f,f_Q\}}^\infty
\int_{Q \cap\mb\{f>\lambda\}} (\Md_{\alpha,Q}\mu)^\frac{1}{q} \intd\mathcal{H}^{n-1}
\intd\lambda \\
&\leq 
C_2 
\int_Q
|\nabla f| \, (\Md_{\alpha,Q}\mu)^\frac{1}{q}
\intd x .
\end{align*}
We have bounded both summands in \cref{split} which means we can conclude that
\[
\biggl( \int_Q \lvert f-f_Q \rvert^q \intd\mu \biggr)^\frac{1}{q} \leq C
\int_Q
|\nabla f| \, (\Md_{\alpha,Q}\mu)^\frac{1}{q}
\intd x ,
\]
where $C = 8 (1+2^{n-1}) C_1 $.
\end{proof}

\section{Examples against weighted $(q,p)$-Poincar\'e inequalities for $p>1$
and against a larger fractional parameter}
\label{sec:counterexample}

In this section,
we prove that the corresponding $L^p$-versions of the weighted fractional and classical Poincar\'e inequalities \cref{thm:wfracpoincare,cor:weight-clas-poincare} do not hold.
This is motivated by ~\cite[Theorems 2.4 \& 2.9]{hurri2022} where sub-optimal results were obtained.

More precisely we show that for every cube \(Q\subset\mathbb R^n\),
$1<p<n$, $p\leq q \leq \frac{np}{n-p}$, $\alpha = n-\frac{q}{p}(n-p)$,
and \(C>0\) there is
a Radon measure $\mu\ll\mathcal{L}$ and a Lipschitz function $f$ with
\begin{equation}
\label{eq:(p_q)-wpoincare}
\biggl( \int_Q \lvert f-f_Q \rvert^q \intd\mu \biggr)^\frac{1}{q}
> C
\biggl(
\int_Q
|\nabla f|^p
\,
(\M_\alpha\mu)^{\frac{p}{q}}
\intd x
\biggr)^\frac{1}{p}
,
\end{equation}
and that for any
$0<\delta<1$, $1<p<\min\bigl\{\frac{n}{\delta},\f1{1-\delta}\bigr\}$, $p\leq q \leq \frac{np}{n-\delta p}$, $\alpha = n-\frac{q}{p}(n-\delta p)$,
and \(C>0\) there is
a Radon measure $\mu\ll\mathcal{L}$ and a Lipschitz function $f$ with
\begin{equation}
\label{eq:(p_q)-wfracpoincare}
\biggl( \int_Q \lvert f-f_Q \rvert^q \intd\mu \biggr)^\frac{1}{q}
>
C
(1-\delta)^\f1p
\biggl(
\int_Q
\int_Q
\f{|f(x)-f(y)|^p}{|x-y|^{n+\delta p}}
\intd y
\, 
(\M_\alpha\mu(x))^\frac{p}{q}
\intd x
\biggr)^\frac{1}{p}
.
\end{equation}
We do not know if the condition \(p<\f1{1-\delta}\) is necessary.

Moreover, we show that given \(q\geq1\) (and \(\delta\)), the value
\(\alpha = n-q(n-1)\)
(or
\(\alpha = n-q(n-\delta)\)
respectively)
for the fractional parameter is the best possible for which \cref{thm:wfracpoincare,cor:weight-clas-poincare} hold in the following sense.
For any \(\varepsilon\geq0\) we have the pointwise inequality
\(
\Md_{\alpha,Q}\mu(x)
\leq
\sle(Q)^\varepsilon
\Md_{\alpha-\varepsilon,Q}\mu(x)
\)
for \(x\in Q\).
Hence, \cref{thm:wfracpoincare,cor:weight-clas-poincare} also hold with \(\sle(Q)^\varepsilon\Md_{\alpha-\varepsilon,Q}\mu(x)\) instead of \(\Md_{\alpha,Q}\mu(x)\).
This argument clearly only works for \(\varepsilon\geq0\).
And indeed, we show that \cref{thm:wfracpoincare,cor:weight-clas-poincare} fail when we replace \(\Md_{\alpha,Q}\mu(x)\) by \(\sle(Q)^{-\varepsilon}\Md_{\alpha+\varepsilon,Q}\mu(x)\) with any \(\varepsilon>0\).
We show this even for the fractional maximal function \(\M_\alpha\mu\) which is larger than \(\Md_{\alpha,Q}\mu\) up to a constant.
More precisely, we show that for any \(1\leq q\leq\f n{n-1}\), \(\alpha = n- q(n-1)\), \(\varepsilon>0\) and \(C>0\)
there is a Radon measure $\mu\ll\mathcal{L}$ 
and a Lipschitz function $f$
\begin{equation}
\label{eq_optimalalpha}
\biggl( \int_Q \lvert f-f_Q \rvert^q \intd\mu \biggr)^\frac{1}{q}
>
C
\f1
{
\sle(Q)^{\f\varepsilon q}
}
\int_Q
|\nabla f|
\,
(\M_{\alpha+\varepsilon}\mu)^\frac{1}{q}
\intd x
,
\end{equation}
and that for any \(0<\delta<1\), \(1\leq q\leq\f n{n-\delta}\), \(\alpha=n-q(n-\delta)\), \(\varepsilon>0\) and \(C>0\)
there is a Radon measure $\mu\ll\mathcal{L}$ 
and a Lipschitz function $f$ with
\begin{equation}
\label{eq_optimalalphafrac}
\biggl( \int_Q \lvert f-f_Q \rvert^q \intd\mu \biggr)^\frac{1}{q}
>
C
\f{
1-\delta
}{
\sle(Q)^{\f\varepsilon q}
}
\int_Q
\int_Q
\f{|f(x)-f(y)|}{|x-y|^{n+\delta}}
\intd y
\, 
(\M_{\alpha+\varepsilon}\mu(x))^\frac1{q}
\intd x
.
\end{equation}

We proceed with the proof of \cref{eq:(p_q)-wpoincare,eq:(p_q)-wfracpoincare,eq_optimalalpha,eq_optimalalphafrac}.
Let $Q_0 \subset B(0,1)$
be the cube with center \(0\) and sidelength \(1/\sqrt{n}\).
By translation and dilation it suffices to find a function \(f\) and a measure \(\mu\) which satisfy \cref{eq:(p_q)-wpoincare,eq:(p_q)-wfracpoincare,eq_optimalalpha,eq_optimalalphafrac} on \(Q_0\).
Consider the sequence of absolutely continuous measures and Lipschitz functions
\[
\mu_k(A) 
=
\frac{\lm{B(0,e^{-k})\cap A}}{\lm{B(0,e^{-k})}}
\qquad\text{and}\qquad
f_k(x)
=
\min\{
-\log |x|
,
k
\}
\]
with \(k\in\mathbb{N}\).
Denote
\[
\int_{Q_0}
|f_k|
\intd x
\leq
\int_{Q_0}
\bigl|\log|x|\bigr|
\intd x
=
c<\infty
.
\]
It holds that $f_k(x)= k$ for $|x| \leq e^{-k}$.
Thus, for \(k\geq c\)
we have
\begin{equation}
\label{eq_oscillation2}
\biggl(
\int_{Q_0}
|f_k-(f_k)_{Q_0}|^q\intd\mu_k
\biggr)^{\f 1q}
\geq
k-c
\end{equation}
for any \(q\geq1\).
Furthermore for any \(0\leq\alpha\leq n\) we have
\[
\M_\alpha \mu_k(x) \leq \frac{C_1}{(|x|+e^{-k})^{n-\alpha}}
\leq
\frac{C_1}{|x|^{n-\alpha}}
\]
for some dimensional constant $C_1$
and
\[
|\nabla f_k(x)|
=
\f1{|x|}
\ind{\{|\cdot|\geq e^{-k}\}}(x)
.
\]
Let \(p>1\), \(q\) and \(\alpha\) be as specified for \cref{eq:(p_q)-wpoincare}. 
Then
\begin{equation}
\label{eq_gradientnorm2}
\begin{split}
\int_{Q_0}
|\nabla f_k|^p
\,
(\M_\alpha\mu_k)^{\f pq}
\intd x
&\leq
\int_{B(0,1)}
|\nabla f_k|^p
\,
(\M_\alpha\mu_k)^{\f pq}
\intd x
\\
&\leq
n\sigma_n
C_1^\frac{p}{q}
\int_0^1
\f1{r^p}
\ind{\{r\geq e^{-k}\}}
r^{(\alpha-n)\f pq}
r^{n-1}
\intd r
\\
&=
C_2
\int_{e^{-k}}^1
r^{-p+(\alpha-n)\f pq+n-1}
\intd r
\\
&=
C_2
\int_{e^{-k}}^1
r^{-p-(n-p)+n-1}
\intd r
\\
&=
C_2
\int_{e^{-k}}^1
\f1r
\intd r
\\
&=
C_2
(\log(1)-\log(e^{-k}))
\\
&=
C_2
k
,
\end{split}
\end{equation}
where
$C_2 = n\sigma_n C_1^\frac{p}{q}$.
Choosing $k\in\mathbb{N}$ large enough, for example
\[
k>(C C_2^{\f1p} +c)^\frac{p}{p-1}
,
\]
we use \cref{eq_oscillation2,eq_gradientnorm2} to conclude \cref{eq:(p_q)-wpoincare} for \(p>1\).

We use the same sequence of functions and measures to satisfy the remaining inequalities \cref{eq:(p_q)-wfracpoincare,eq_optimalalpha,eq_optimalalphafrac}.
In order to find \(k\) such that \(\mu_k,f_k\) satisfies \cref{eq_optimalalpha},
let $p$, \(\alpha\), \(q\) and \(\varepsilon>0\) be as specified there.
Then we have
\begin{equation}
\label{eq_gradientnormalpha}
\begin{split}
\f1{\sle(Q_0)^{\frac\varepsilon q}}
\int_{Q_0}
|\nabla f_k|
\,
(\M_{\alpha+\varepsilon} \mu_k)^{\f 1q}
\intd x
&\leq
n^{\f\varepsilon{2q}}
\int_{B(0,1)}
|\nabla f_k|
\,
(\M_{\alpha+\varepsilon}\mu_k)^{\f 1q}
\intd x
\\
&\leq
n^{1+\f{\varepsilon}{2q}}\sigma_n
C_1^\frac{1}{q}
\int_0^1
\f1{r}
\ind{\{r\geq e^{-k}\}}
r^{(\alpha+\varepsilon-n)\f 1q}
r^{n-1}
\intd r
\\
&=
C_3
\int_{e^{-k}}^1
r^{(\alpha+\varepsilon-n)\f 1q+n-2}
\intd r
\\
&=
C_3
\int_{e^{-k}}^1
r^{\f\varepsilon q-1}
\intd r
\\
&=
\frac{C_3q}{\varepsilon}
(1-e^{-\f\varepsilon q k})
,
\end{split}
\end{equation}
where \(C_3=n^{1+\f{\varepsilon}{2q}}\sigma_nC_1^\frac{1}{q}\).
Choosing $k\in\mathbb{N}$ large enough, for example
\[
k > \frac{2^{\f\varepsilon q} CC_3q}{\varepsilon}+c
,
\]
we use \cref{eq_oscillation2,eq_gradientnormalpha} to conclude \cref{eq_optimalalpha} for \(\varepsilon>0\).

In order to find \(k\) such that \(\mu_k,f_k\) satisfies \cref{eq:(p_q)-wfracpoincare,eq_optimalalphafrac}, we first note that
by for example~\cite[Lemma~3.6]{hurri2022},
both \((\M_\alpha\mu)^{\f pq}\) and \((\M_{\alpha+\varepsilon}\mu)^{\f1q}\) are \(A_1\)-weights with
\begin{align*}
\bigl[(\M_\alpha\mu)^{\f pq}\bigr]_{A_1}
&\leq
\f{
15^n4n
}{
p\delta
}
\\
\bigl[(\M_{\alpha+\varepsilon}\mu)^{\f1q}\bigr]_{A_1}
&\leq
\f{
15^n4n
}{
\delta
+
\f{(n-\delta)}{n-\alpha}\varepsilon
}
.
\end{align*}
Since by assumption we have \(\delta>0\) or \(\delta>\f{p-1}p\) respectively,
we can apply \cref{cor:frac-clas-A1} with
\(
w=
(\M_\alpha\mu)^{\f pq}
\) and 
\(
w=
(\M_{\alpha+\varepsilon}\mu)^{\f1q}
,
\) 
and we conclude \cref{eq:(p_q)-wfracpoincare,eq_optimalalphafrac} from \cref{eq:(p_q)-wpoincare,eq_optimalalpha} respectively.

\end{document}